\documentclass[10pt,a4paper,final]{article}
\pdfoutput=1
\usepackage[latin1]{inputenc}
\usepackage{amsmath}
\usepackage{amssymb}
\usepackage{graphicx}
\usepackage{hyperref}
\usepackage{enumerate}
\usepackage{amsthm}
\usepackage{todo}
\usepackage{amssymb}
\usepackage{amsmath}
\usepackage{latexsym}
\usepackage{ragged2e}
\usepackage{amssymb}
\usepackage{amsthm}
\usepackage{cite}
\usepackage{mathrsfs}
\usepackage{bbm}
\usepackage{bbold}
\usepackage{subfigure}
\usepackage{float}
\usepackage{epsfig}
\usepackage[numbers]{natbib}
\usepackage{mathtools}
\usepackage{listings}
\usepackage{enumerate}
\usepackage{stmaryrd}
\usepackage{amsthm}
\usepackage{cite}
\usepackage{mathrsfs}
\usepackage{bbm}
\usepackage{cite}
\usepackage{mathrsfs}
\usepackage{bbm}
\usepackage{graphicx}
\usepackage{caption}
\usepackage{indentfirst}
\usepackage{textcomp}
\usepackage{afterpage}
\usepackage{subfigure}
\usepackage{float}
\usepackage{epsfig}
\usepackage[numbers]{natbib}
\usepackage{mathtools}
\usepackage{listings}
\usepackage{pgf}
\usepackage{listings}
\usepackage{epstopdf}
\usepackage[notcite,notref]{showkeys}
\usepackage[hmargin=3.5cm, vmargin=3.5cm]{geometry}
\usepackage{stmaryrd}
\usepackage{caption}

\theoremstyle{plain} 
\newtheorem{thm}{Theorem}[section] 
\newtheorem{prop}[thm]{Proposition}

\newtheorem{lem}[thm]{Lemma}

\theoremstyle{definition} 
\newtheorem{defn}[thm]{Definition}

\theoremstyle{remark} 
\newtheorem{oss}[thm]{Remark}

\newcommand{\ind}{\mathbbm{1}}

\newcommand{\N}{\mathbb{N}}

\newcommand{\CC}{\mathcal{C}}

\DeclareMathOperator{\Bin}{Bin}

\title{A simple path to component sizes in critical random graphs}
\author{Umberto De Ambroggio\thanks{University of Munich, Department of Mathematics, Theresienstr. 39, 80333 Munich, Germany. \texttt{umbidea@gmail.com}}\hspace{1.5mm}}
\begin{document}

\maketitle
 
\begin{abstract}
We describe a robust methodology, based on the martingale argument of Nachmias and Peres and random walk estimates, to obtain simple upper and lower bounds on the size of a maximal component in several random graphs \textit{at criticality}. Even though the main result is not new, we believe the the material presented here is interesting because it unifies several proofs found in the literature into a common framework. More specifically, we give easy-to-check conditions that, when satisfied, allow an immediate derivation of the above mentioned bounds. 
\end{abstract}

\section{Introduction}\label{sectionintro}
The scope of this work is to describe a simple methodology, based on the martingale argument introduced by Nachmias and Peres \cite{nachmias_peres:CRG_mgs} as well as a ballot-type estimate from \cite{de_ambroggio_roberts:near_critical_ER}, to derive upper and lower bounds on the size of a maximal connected component in several models of random graphs when considered \textit{at} criticality. 

Throughout the introduction we write $\mathcal{C}_{\max}$ for a largest component; later on we will use the notation $\mathcal{C}_{\max}(\mathbb{G})$ to denote a largest component in a random (multi-)graph $\mathbb{G}$. 

We emphasise that the bounds on $|\mathcal{C}_{\max}|$ in the (critical) random graphs considered here are not new (perhaps with the exception of Theorem \ref{mainthm3}-(c) which, to the best of our knowledge, did not appear before), but we believe the material presented here to be interesting because it highlights a common proof strategy which works for many random graphs when considered at criticality of their parameters.

More specifically, we highlight the key aspects of the martingale argument of Nachmias and Peres \cite{nachmias_peres:CRG_mgs} by providing easy-to-check conditions that, when satisfied, allow a quick derivation of \textit{lower} bounds on $|\mathcal{C}_{\max}|$; moreover, we further simplify the argument described in \cite{de_ambroggio:component_sizes_crit_RGs} to derive \textit{upper} bounds on $|\mathcal{C}_{\max}|$.

All in all, we provide a simple recipe to obtain lower and upper bounds on $|\mathcal{C}_{\max}|$ in \textit{critical} random graphs, shortening and unifying several distinct proofs found in the literature into a common framework.

Let us proceed by introducing the models considered in this work.\\

\textbf{The $\mathbb{G}(n,p)$ model}. The Erd\H{o}s-R\'enyi random graph, denoted by $\mathbb{G}(n,p)$ is the random graph on $n$ vertices obtained by performing $p$-bond percolation on $\mathbb{K}_n$, the complete graph with vertex set $[n]\coloneqq \{1,\dots,n\}$; that is, each edge of $\mathbb{K}_n$ is \textit{independently} retained with probability $p$ and deleted with probability $1-p$.\\

\textbf{The $\mathbb{G}(n,d,p)$ model}. Let $d\ge 3$ be a fixed integer, and let $n\in \mathbb{N}$ be such that $dn$ is even. We let $\mathbb{G}(n,d)$ be a (random, simple) $d$-regular graph sampled uniformly at random from the set of all simple $d$-regular graphs on $[n]$, and then denote by $\mathbb{G}(n,d,p)$ the random graph obtained by performing $p$-bond percolation on a realisation of $\mathbb{G}(n,d)$.\\

\textbf{The $\mathbb{G}(n,m,p)$ model}. Given positive integers $n,m\in \mathbb{N}$, let $V\coloneqq \{v_1,\dots,v_n\}$ and $W\coloneqq \{w_1,\dots,w_m\}$. The random bipartite graph $\mathbb{B}(n,m,p)$ is obtained by performing $p$-bond percolation on $\mathbb{K}(n,m)$, the latter denoting the complete bipartite graph with bipartition $(V,W)$ where each node $v_i$ is linked to each $w_j$ (for $i\in [n],j\in [m]$). The random intersection graph $\mathbb{G}(n,m,p)$ is the (random) graph with vertex set $V$ in which each pair of distinct vertices $\{v_i,v_j\}$ is present as an edge if $v_i$ and $v_j$ share \textit{at least one} common neighbour in $\mathbb{B}(n,m,p)$. We call $V$ the set of vertices, whereas we call $W$ the set of attributes (or features or auxiliary vertices). \\

\textbf{The $\mathbb{G}(n,\beta,\lambda)$ model}. Here we give an informal definition of the model; the rigorous formulation (which is taken from \cite{dembo_et_al:component_sizes_quantum_RG}) can be found in Subsection \ref{subsectionGnbetalambda}. We have $n$ vertices, which are represented by circles of length $\beta>0$, punctured with \textit{independent} Poisson point processes of intensity $\lambda>0$. Because of these (Poisson) processes of holes, circles can be written as \textit{finite} \textit{disjoint} unions of connected intervals. The edges of the graph link intervals on distinct vertices (circles) and they are generated as follows. Given any two (distinct) vertices $u\neq v$, we run another \textit{independent} Poisson point process of intensity $1/n$ on a (third) circle of length $\beta>0$. If such a point process jumps at a time which is contained in both an interval of $u$ and an interval of $v$ (recall the decomposition mentioned earlier), then these two intervals are considered directly connected.\\

It is well known (see e.g. \cite{bollobas_book}, \cite{janson_et_al:random_graphs} and \cite{remco:random_graphs}) that the $\mathbb{G}(n,\gamma/n)$ random graph undergoes a phase transition as $\gamma$ passes $1$. Specifically, if $\gamma \leq 1-\varepsilon$ for some constant $\varepsilon\in (0,1)$, then $|\mathcal C_{\max}|$ is of order $\log n$; if $\gamma=1$ (\textit{critical} case), then $|\mathcal C_{\max}|$ is of order $n^{2/3}$; and if $\gamma \geq 1+\varepsilon$ for some constant $\varepsilon>0$, then $|\mathcal C_{\max}|$ is of order $n$. Luczak \cite{L1} and Luczak, Pittel and Wierman \cite{L2} established the existence of a \textit{scaling-window} of width $n^{-1/3}$ around the (critical) point $1/n$: for all $p=p(n)$ of the form $p=n^{-1/3}(1+O(n^{-1/3}))$ the random variable $|\mathcal{C}_{\max}|n^{-2/3}$ converges in distribution to a non-trivial random variable and, in particular, is not concentrated. Furthermore, when $p=(1+\varepsilon)/n$ where $\varepsilon=\varepsilon(n)$ satisfies $|\varepsilon|\ll 1$ but $|\varepsilon|^3 n\rightarrow \infty$, the random
variable $|\mathcal{C}_{\max}|$ is concentrated around some known value.

We refer the reader to the monographs \cite{bollobas_book}, \cite{remco:random_graphs} and \cite{janson_et_al:random_graphs} for a thorough study of the $\mathbb{G}(n,p)$ random graph. 

A similar phenomenon occurs in the $\mathbb{G}(n,d,p)$ model. Goerdt \cite{goerdt} showed that also the $\mathbb{G}(n,d,\gamma/(d-1))$ random graph undergoes a phase transition as $\gamma$ passes $1$: specifically, $|\mathcal C_{\max}|$ is of order $\log(n)$ when $\gamma \leq 1-\varepsilon$ for some constant $\varepsilon\in (0,1)$, whereas it is of order $n$ when $\gamma \geq 1+\varepsilon$ for some constant $\varepsilon>0$ (another proof was provided by Alon, Benjamini and Stacey \cite{alon_benj_stacey}).

Nachmias and Peres \cite{nachmias_peres:CRG_mgs} introduced a martingale argument to analyse the near-critical behavior of the $\mathbb{G}(n,p)$ model, and subsequently employed similar techniques to prove that also the $\mathbb{G}(n,d,p)$ random graph exhibits a scaling window of order $n^{-1/3}$ around the critical point $p=(d-1)^{-1}$.

Amongst other results, in \cite{nachmias_peres:CRG_mgs} the authors proved that, in the $\mathbb{G}(n,1/n)$ model, for any large enough $A$ and $n\in \mathbb{N}$ we have
\begin{equation}\label{bounder}
	\mathbb{P}\left(|\mathcal C_{\max}|< n^{2/3}/A\right)=O(A^{-3/5}) \text{ and }\mathbb{P}\left(|\mathcal C_{\max}|>An^{2/3}\right)= O(A^{-3/2}).
\end{equation}
(They also prove an exponential bound for the probability on the right-hand side of (\ref{bounder}).)
On the other hand, concerning the $\mathbb{G}(n,d,1/(d-1))$ model, the results in \cite{nachmias:critical_perco_rand_regular} imply that, for any (fixed) $d\geq 3$, there exists a positive constant $D=D(d)$ such that, for $\delta>0$ small enough and all sufficiently large $n$, then
	\begin{equation}\label{boundperc}
	\mathbb{P}\left(|\mathcal C_{\max}|< n^{2/3}/A\right)=O(A^{-1/2}) \text{ and }\mathbb{P}\left(|\mathcal C_{\max}|>An^{2/3}\right)= O(A^{-3/2}).
	\end{equation}
	(They also prove an exponential bound for the probability on the right-hand side of (\ref{boundperc}).)
The martingale argument of Nachmias and Peres is particularly robust and it was subsequently adapted by several other authors to study different random graph models near criticality of their parameters.

For example, Dembo, Levit and Vadlamani \cite{dembo_et_al:component_sizes_quantum_RG} adapted such a methodology to study the near-critical behaviour of the \textit{quantum} random graph that we defined earlier. An equivalent description of the $\mathbb{G}_{n,\beta,\lambda}$, that we use in this paper, goes as follows. We have $n$ circles (nodes) of length $\theta\coloneqq \beta \lambda>0$ with (independent) rate $1$ Poisson process of holes on them. In this new setting, we use (independent) Poisson processes of intensity $(\lambda n)^{-1}$ for creating links between each pair of punctured circles. 

The critical parameter (curve) for the quantum random graph in the $(\beta,\lambda)$-parameter space was identified by Ioffe and Levit \cite{ioffe_levit} by comparison with a critical branching process whose offspring distribution is the so-called \textit{cut-gamma} law $\Gamma_{\theta}(2,1)$, which is the distribution of $J\coloneqq (J_1+ J_2 ) \wedge \theta$ for independent random variables $J_1,J_2$ with the $\text{Exp}(1)$ law. Setting $F(\lambda \beta)\coloneqq 2(1-e^{-\lambda \beta})-\lambda \beta e^{-\lambda\beta}$, criticality is reached when
\[F(\beta,\lambda)\coloneqq \lambda^{-1}F(\lambda \beta)=1, \]
with $\lambda^{-1}F(\lambda \beta)$ corresponding to the expected length of an interval $I$ in the quantum random graph. It was shown in \cite{ioffe_levit} that, if $F(\beta,\lambda)>1$, then a giant component of order $\Theta(n)$ emerges, whereas when $F(\beta,\lambda)<1$ all components are typically of order $O(\log(n))$ (see \cite{dembo_et_al:component_sizes_quantum_RG} and \cite{ioffe_levit}).

Amongst other results, it was shown in \cite{dembo_et_al:component_sizes_quantum_RG} that 
\begin{equation}\label{boundgnbetalambda}
\mathbb{P}\left(|\mathcal C_{\max}|< n^{2/3}/A\right)=O(A^{-3/5}) \text{ and }\mathbb{P}\left(|\mathcal C_{\max}|>An^{2/3}\right)= O(A^{-3/2}).
\end{equation}
The same methodology was later used by Hatami and Molloy \cite{hatami_molloy_conf} to identify the scaling window for a random graph with a given degree sequence, and also by the author and Pachon \cite{de_ambroggio_pachon:upper_bounds_inhom_RGs} to study the critical behaviour in the Norros-Reittu model (with the purpose of providing more precise probabilistic estimates compared to those established by van der Hofstad \cite{hofstad_critic}, who was the first to provide a complete picture of the component structure in (near-)critical rank-1 inhomogeneous random graphs).

The component structure in the $\mathbb{G}(n,m,p)$ random graph was analysed by Behrisch \cite{ber} for the case $m=m(n)=n^{\alpha}$ and $p^2m=c/n$, with $\alpha,c\in (0,1)\cup (1,\infty)$. 

As shown by Stark \cite{stark}, the vertex degree distribution (i.e. the distribution of the degree of a vertex selected uniformly at random) is highly dependent on the value of $\alpha$. However, as shown by Deijfen and Kets \cite{deijfen_kets_2009}, the clustering is controllable only when $\alpha=1$. 

This case was subsequently considered by Lageras and Lindholm \cite{laglind}. Specifically, in \cite{laglind} the authors considered the $\mathbb{G}(n,m,p)$ random graph having $m=\lfloor \beta n\rfloor$ and $p=\gamma/n$ (with $\gamma,\beta>0$ parameters of the model) and proved that the $\mathbb{G}(n,m,p)$ model undergoes a phase transition as $\beta \gamma^2$ passes $1$. Indeed, setting $\mu= \beta \gamma^2$, they proved that if $\mu<1$ (sub-critical case) then with probability tending to one there is no component in $\mathbb{G}(n,m,p)$ with more than $O(\log(n))$ vertices, while if $\mu >1$ (super-critical case) then, with probability tending to one, there exists a unique giant component of size $n\delta$ where $\delta\in (0,1)$, and the size of the second largest component is at most of order $\log(n)$.

It was then shown in \cite{de_ambroggio:component_sizes_crit_RGs} that, in the (critical) case $\mu=1$, then 
\begin{equation}\label{boundgnmp1}
\mathbb{P}\left(|\mathcal C_{\max}|>An^{2/3}\right)= O(A^{-3/2}).
\end{equation} 
Here we also show that, in the critical case, it is unlikely for $\mathcal C_{\max}$ to contain significantly less than $n^{2/3}$ vertices; that is, we show that
\begin{equation}\label{boundgnmp2}
\mathbb{P}\left(|\mathcal C_{\max}|<n^{2/3}/A\right)= O(A^{-1/2}).
\end{equation} 

Our goal here is to \textit{unify} the proofs of the above results by providing a common simple recipe that has the potential to be applied also to other critical random graphs.
	
The argument relies on the following three basic ingredients.
	\begin{itemize}
		\item An \textbf{exploration process}, which is a (standard) algorithmic procedure to sequentially reveal the component structure of the underlying random graph and that reduces the study of component sizes to the analysis of the trajectory of a non-negative random process; the intuitive description of an exploration process is provided below.
		\item An upper bound on the probability that \textit{all} the excursions of a non-negative random process (satisfying some basic moment conditions) last less than a given number of steps; this is the content of \textbf{Proposition \ref{jointprop}}, whose elementary proof uses the Optional Stopping Theorem along the lines of Nachmias and Peres \cite{nachmias_peres:CRG_mgs,nachmias:critical_perco_rand_regular}.
		\item A simple bound for the probability that a random walk (with iid increments having mean zero and finite second moment) remains positive for a given number of steps; this is the content of \textbf{Lemma \ref{simplelem}}, whose very short proof solely relies on a ballot-type estimate (Lemma \ref{lemmaballot} below) introduced in \cite{de_ambroggio_roberts:near_critical_ER}.
	\end{itemize} 
Before proceeding, let us describe the rough idea of an exploration process (see e.g. \cite{nachmias_peres:CRG_mgs,nachmias:critical_perco_rand_regular}, \cite{de_ambroggio_roberts:near_critical_ER} and references therein).\\

\textbf{Exploration processes: the basic principle}. In our context, an exploration process is an algorithmic procedure used to sequentially reveal the structure of the underlying random (multi-)graph. Roughly speaking, it works as follows. Initially, that is at time $t=0$, one of the vertices, say $u$, is declared \textit{active}, whereas all the remaining nodes are \textit{unseen}, and there are no \textit{explored} vertices. At time $t=1$, we reveal the unseen neighbours of $u$, which are declared active. The step terminates by declaring $u$ explored. Then we iterate: at time $t=2$, we select one of the active nodes (if any), say $u'$, and reveal its unseen neighbours, which are declared active, whereas $u'$ itself becomes explored. If at some step $t-1$ the set of active vertices becomes empty, then we pick (deterministically or in some random fashion) one of the unseen vertices, say $v$, reveal its unseen neighbours, which are declared active, and change the status of $v$ to explored. The procedure terminates when all the vertices have been explored. (We remark that, in an exploration process, it is not always the vertices that are being explored; indeed, according to the model under investigation, we might decide to reveal edges.) Summarizing, at each step of the procedure we \textit{add} some random number of vertices to the set of active nodes and (deterministically) \textit{remove} from such a set the node whose unseen neighbours were revealed during that step.
Denoting by $\mathbb{G}$ an undirected random graph, we use the above procedure to sequentially reveal the connected components of $\mathbb{G}$. Let us write $\eta_t$ for the random number of (unseen) vertices that become active during step $t$. Then, writing $R_i$ for the number of active vertices at the end of step $i$ in the procedure, if $R_{t-1}\geq 1$ we can express $R_t$ as $R_t=R_{t-1}+\eta_t-1$: the number of active vertices at the end of step $t$ equals the number of active nodes at the end of step $t-1$ plus the number of unseen vertices that become active at time $t$, minus the (active) node which becomes explored (and whose unseen neighbours were revealed during step $t$). On the other hand, if $R_{t-1}=0$ (meaning that the are no longer active nodes at the end of step $t-1$), then $R_t=\eta_t$, since we are not removing any active item at the end of step $t$. Thus it becomes clear that one can approximate the number of vertices in a connected component by the number of consecutive steps at which the process $R_t$ stays \textit{positive}.\\

We can use an exploration process to bound (from above) the probability that $\CC_{\max}(\mathbb{G})$ contains less than $T\in \mathbb{N}$ vertices by the probability that \textit{all} the positive excursions of $R_t$ never last for more than $T$ steps; similarly, the probability that the component of a node (from which we start the exploration process) contains more than a given number of vertices, say $T$, can be related to the probability that $R_t$ remains positive for \textit{at least} $T$ steps.

More specifically, we will use exploration processes together with Proposition \ref{jointprop} of the next section to give very short and simple proofs of the fact that, at criticality, the random graphs introduced early have largest components that contain more than $n^{2/3}/A$ vertices with probability $1-O(A^{-1/2})$. Indeed, we have the following 
\begin{thm}\label{mainthm3}
	Let $1< A=A(n) < n^{2/3}$. There exist constants $n_0\in \mathbb{N}$ and $A_0\geq 1$ for which the following statements hold true. 
	\begin{itemize}
		\item [(a)]Let $p=p(n)=1/n$. There exists a constant $C_1>0$ such that, for all $n\geq n_0$ and for every $A\geq A_0$,
		\begin{equation}\label{gnp}
		\mathbb{P}(|\CC_{\max}(\mathbb{G}(n,p))|<n^{2/3}/A)\leq C_1/A^{1/2}.
		\end{equation}
		\item [(b)] Let $d\geq 3$ (fixed) and $p=p(d)=1/(d-1)$. There exists a constant $C_2=C_2(d)>0$ which depends on $d$ such that, for all $n\geq n_0$ and for every $A\geq A_0$,
		\begin{equation}\label{gndp}
		\mathbb{P}(|\CC_{\max}(\mathbb{G}(n,d,p))|<n^{2/3}/A)\leq C_2/A^{1/2}.
		\end{equation}
		\item [(c)] Let $p=p(n,m)=1/\sqrt{nm}$ and $m=\lfloor \beta n\rfloor$. There exists a constant $C_3=C_3(\beta)>0$ which depends on $\beta$ such that, for all $n\geq n_0$ and for every $A\geq A_0$,
		\begin{equation}\label{gnmp}
		\mathbb{P}(|\CC_{\max}(\mathbb{G}(n,m,p))|<n^{2/3}/A)\leq C_3/A^{1/2}.
		\end{equation}
		\item [(d)] Let $F(\beta,\lambda)=1$. There exists a constant $C_4=C_4(\beta,\lambda)>0$ which depends on $\beta$ and $\lambda$ such that, for all $n\geq n_0$ and for every $A\geq A_0$,
		\begin{equation}\label{gnbetalambda}
		\mathbb{P}(|\CC_{\max}(\mathbb{G}(n,\beta,\lambda))|<n^{2/3}/A)\leq C_4/A^{1/2}.
		\end{equation}
	\end{itemize}
\end{thm}
\begin{oss}
	We remark that our assumption on $A$, which in particular it is allowed to depend on $n$, is not restrictive. Indeed, if $A\geq n^{2/3}$, then $n^{2/3}/A\leq 1$ and hence we obtain that $\mathbb{P}\left(|\CC_{\max}|<n^{2/3}/A\right)= \mathbb{P}\left(|\CC_{\max}|=0\right)=0$, because the component of any given vertex $u\in [n]$ contains at least $u$, and hence $|\CC_{\max}|\geq 1$ by definition. We also note that, since we allow $A$ in (\ref{gndp}) and (\ref{gnp}) to depend on $n$, our result provides information about the whole lower tail of $|\CC_{\max}|$ in each model. Finally, we also emphasize that we could have been precise about the constant $A_0$, by computing an exact value; however, to shorten the computations, we refrained to do so.
\end{oss}
Our second result, whose simple proof is based on Lemmas \ref{simplelem} and \ref{secondsimplelem}, shows that the random graphs considered in Theorem \ref{mainthm3} have largest components that contain less than $An^{2/3}$ vertices with probability $1-O(A^{-3/2})$. 

Concerning bond percolation on a $d$-regular graph we can be more general. It was shown in \cite{nachmias:critical_perco_rand_regular} that the random graph $\mathbb{G}_d(p)$ obtained through critical percolation on \textit{any} $d$-regular graph $\mathbb{G}_d$ with $n$ vertices (with $d\geq 3$ fixed) and $p= (d-1)^{-1}$, have largest components that are unlikely to contain much more than $n^{2/3}$ nodes. We provide an elementary proof of this fact, namely that a largest component in $\mathbb{G}_d(p)$ contains less than $An^{2/3}$ nodes with probability at least $1-c_dA^{-3/2}$, where $c_d$ is an explicit constant that depends on $d$, which is allowed to depend on $n$. (A short proof was also given in  \cite{de_ambroggio:component_sizes_crit_RGs} assuming $d\geq 4$.) 

\begin{thm}\label{corollary}
	There exist $n_0\in \mathbb{N}$ such that, for every $A\geq 1$ and for all $n\geq n_0$, the following statements hold true.
	\begin{itemize}
		\item [(a)] If $p=p(n)= n^{-1}$, then $\mathbb{P}(|\mathcal{C}_{\max}(\mathbb{G}(n,p))|>An^{2/3})\leq C_5A^{-3/2}$, for some constant $C_5>0$.
		\item [(b)] If $p=p(d)= (d-1)^{-1}$, then $\mathbb{P}(|\CC_{\max}(\mathbb{G}_d(p))|>An^{2/3})\leq C_6A^{-3/2}$, for some constant $C_6=C_6(d)>0$ which depends on $d$.
		\item [(c)] If $p=p(n,m)=1/\sqrt{nm}$ and $m=\lfloor \beta n\rfloor$,  then $\mathbb{P}(|\mathcal{C}_{\max}(\mathbb{G}(n,m,p))|>An^{2/3})\leq C_7A^{-3/2}$ for some constant $C_7=C_7(\beta)>0$ which depends on $\beta$.
		\item [(d)] If $F(\beta,\lambda)$, then $\mathbb{P}(|\mathcal{C}_{\max}(\mathbb{G}(n,\beta,\lambda))|>An^{2/3})\leq C_8A^{-3/2}$, for some constant $C_8=C_8(\beta,\lambda)>0$ which depends on $\beta,\lambda$.
	\end{itemize}
\end{thm}
\begin{oss}
	A short proof of (some of) the bounds given in Theorem \ref{corollary} was given in \cite{de_ambroggio:component_sizes_crit_RGs}. More precisely, Theorem 1 in \cite{de_ambroggio:component_sizes_crit_RGs} said the following. Let $V_n$ denote a node selected uniformly at random from the vertex of a random graph $\mathbb{G}$. Suppose that it is possible to bound from above the probability that $\mathcal{C}(V_n)$ contains more than $k\in \mathbb{N}$ vertices by the probability that a random walk with iid increments remains positive for $k$ steps. (Such an assumption is motivated by the use of an exploration process to reveal the connected components of $\mathbb{G}$). Then, if the increments of the random walk have a moment generating function satisfying some specific conditions, we can conclude that $|\mathcal{C}_{\max}(\mathbb{G})|\leq An^{2/3}$ with probability at least $1-O(A^{-3/2})$. This was established making use of a ballot-type result (Lemma \ref{lemmaballot} of the next section) introduced in \cite{de_ambroggio_roberts:near_critical_ER}. The proof that we give here follows the same line, but it is even shorter. \\ 
\end{oss}


\textbf{Heuristic derivation of the bounds}. The proofs of all the inequalities stated in Theorems \ref{mainthm3} and \ref{corollary} start by describing an exploration process to sequentially generate the connected components of the specific (random) graph under investigation (or some closely related model) and which reduces the study of components sizes to the analysis of the trajectory of a non-negative random process. 

Let's start by describing how to obtain the bounds stated in Theorem \ref{mainthm3}. (In what follows, the symbols $\lesssim$ and $\ll$ are used rather imprecisely; in particular, we write $a \lesssim b$ and $a\ll b$ if $a$ is at most a constant times $b$ and $a$ is much smaller than $b$, respectively.)

Recall the earlier description of an exploration process, which we use to relate the probability that $\CC_{\max}(\mathbb{G})$ contains less than $T$ vertices (where $T=T(n)\in \mathbb{N}$) with the probability that all the \textit{positive} excursions of $R_t$ never last for more than $T$ steps. 
 
Setting $t_0\coloneqq 0$ and defining $t_i$ as the first time $t\geq t_{i-1}+1$ at which $R_t=0$, we can bound
\begin{equation}\label{firststep}
\mathbb{P}(|\CC_{\max}(\mathbb{G})|<T)\lesssim \mathbb{P}(t_i-t_{i-1}<T\text{ }\forall i).
\end{equation}
Now, proceeding as in \cite{nachmias_peres:CRG_mgs,nachmias:critical_perco_rand_regular}, the idea is to show that, with sufficiently high probability, the process $R_t$ reaches some (large enough) level $h=h(n)$ before some time $\mathbb{N} \ni T'=T'(n)\ll n$ and then remains positive for at least $T$ steps. We now explain why we should expect both events to be likely (provided $h,T$ and $T'$ are chosen in a suitable manner).

The basic idea is that the process $R_t$ behaves like a \textit{proper} random walk, where we say that a random walk is proper if its iid increments have mean zero and a finite second moment. Since the latter process has a strictly positive probability to be, at time $T'$, at height $\Theta(\sqrt{T'})$, by taking $h\ll \sqrt{T'}$ it becomes unlikely for $R_t$ to remain below $h$ for $T'$ steps. In other words, asking $R_t$ to remain below $h\ll \sqrt{T'}$ for $T'$ consecutive steps is like asking a proper random walk to be far below where it \textit{could} be, an event which occurs with small probability.

At the same time, if $\sqrt{T}\ll h$, then it becomes unlikely for $R_t$ (when starting from $h$) to reach zero in less than $T$ steps, since in this case we would be asking a proper random walk to be far below where it \textit{should} be, which is again an event that occurs with small probability.

Thus, defining $\tau_h$ as the minimum between the first time at which $R_t$ is above $h$ and $T'$, noticing that $\tau_h<T'$ implies $R_{\tau_h}\geq h$ we arrive at
\begin{equation}\label{splitproofideas}
\mathbb{P}(t_i-t_{i-1}<T\text{ }\forall i)\leq \mathbb{P}(\tau_h=T')+\mathbb{P}(t_i-t_{i-1}<T\text{ }\forall i,R_{\tau_h}\geq h),
\end{equation}
where we take $\sqrt{T}\ll h\ll \sqrt{T'}$ for the reasons that we mentioned earlier.
With Propositions \ref{NEWmainprop} and \ref{secondmainprop} below, whose proofs are based on the Optional Stopping Theorem along the lines of \cite{nachmias_peres:CRG_mgs,nachmias:critical_perco_rand_regular}, we show that, if the $\eta_i$ satisfy some basic moments conditions, 
then the first probability on the right-hand side of (\ref{splitproofideas}) is $\lesssim h^2/T'$, which is small (since $h\ll \sqrt{T'}$), whereas the second probability in (\ref{splitproofideas}) is $\lesssim T/h^2$, which is small too (as $\sqrt{T}\ll h$). 

Concerning the result stated in Theorem \ref{corollary}, we again use an exploration process to bound from above the probability that $\mathcal{C}(V_n)$ contains more than $k$ nodes by the probability that a random walk with iid increments (having mean zero and finite variance) stays positive for $k$ steps. Then, using a simple ballot-type estimate (namely, Lemma \ref{lemmaballot}) from \cite{de_ambroggio_roberts:near_critical_ER}), we give an elementary proof of the well-know fact that such a probability is $O(k^{-1/2})$. The bounds for $|\mathcal{C}_{\max}|$ then follows from a standard argument which uses Markov's inequality applied to the random variable $Z_{>k}$, counting the number of vertices contained in components formed by more than $k$ nodes. \\

\textbf{Notation}. We write $\mathbb{N}=\{1,2,\dots,\}$ and set $\mathbb{N}_{\infty}\coloneqq \mathbb{N}\cup\{\infty\}, [n]\coloneqq \{1,\ldots,n\}$ for $n\in \mathbb{N}$. Given two sequences of non-negative real numbers $(x_n)_{n\geq 1}$ and $(y_n)_{n\geq 1}$ we write: (i) $x_n=O(y_n)$ if there exist $N\in \mathbb{N}$ and $C\in [0,\infty)$ such that $x_n\leq C y_n$ for all $n\geq N$; (ii) either $x_n=o(y_n)$ or $x_n\ll y_n$ if $x_n/y_n\rightarrow 0$ as $n\rightarrow \infty$; and (iii) either $x_n=\Theta(y_n)$ or $x_n\asymp y_n$ if $x_n=O(y_n)$ and $y_n=O(x_n)$. Sometimes we write $O_d(\cdot),\Theta_d(\cdot)$ to indicate that the constants involved depend on a specific quantity $d$. Given any $a,b\in \mathbb{R}\cup\{\pm\infty\}$, we write $a\vee b$ and $a\wedge b$ for the maximum and the minimum between $a$ and $b$, respectively. The abbreviation iid means \textit{independent and identically distributed}. We write $\text{Ber}(\cdot),\text{Bin}(\cdot,\cdot)$ and $\text{Poi}(\cdot)$ to denote the Bernoulli, binomial, and Poisson distributions, respectively. Let $\mathbb{G}=(V,E)$ be any (undirected, possibly random) multigraph. Given two vertices $u,v\in V$, we write $u\sim v$ if $\{u,v\}\in E$ and say that vertices $u$ and $v$ are \textit{neighbours}. We often write $uv$ as shorthand for the edge $\{u,v\}$. We write $u\leftrightarrow v$ if there exists a path connecting vertices $u$ and $v$, where we adopt the convention that $v\leftrightarrow v$ for every $v\in V$. We denote by $\CC(v)\coloneqq \{u\in V:u\leftrightarrow v\}$ the component containing vertex $v\in V$. We define $\CC_{\max}$ to be some cluster $\CC(v)$ for which $|\CC(v)|$ is maximal, so that $|\CC_{\max}|=\max_{v\in V}|\CC(v)|$. We often write $\CC_{\max}(\mathbb{G})$ to denote a largest component in $\mathbb{G}$.

\section{Preliminaries}
Here we collect the relevant tools that are needed to prove Theorems \ref{mainthm3} and \ref{corollary}. More specifically, in Subsection \ref{sublowertail} we provide an upper bound for the probability that \textit{all} positive excursions of a discrete-time non-negative random process last for less than a given number of steps (that is, the probability which appears on the left-hand side of (\ref{splitproofideas})). Subsequently, in Subsection \ref{subuppertail}, we provide a very short and simple proof of the well-known fact \cite{rwbook} that a random walk with iid increments having zero mean and finite variance stays above zero for $k$ steps with probability at most $ck^{-1/2}$, for some explicit constant $c>0$ (which depends on the distribution of the increments of the random walk).

\subsection{Preliminary result for Theorem \ref{mainthm3}: the martingale method} \label{sublowertail}
In this subsection we consider a non-negative (discrete-time) random process $(R_t)_{t\in [K]\cup \{0\}}$ (where $K\in \mathbb{N}$) which satisfies the same type of recursion that arises when using an exploration process to reveal the connected components of a random graph, and we analyse its \textit{positive} excursions. 

More specifically our goal is to show that, under suitable conditions on the process increments, there is at least one excursion lasting for sufficiently many steps. To achieve this we make use of the Optional Stopping Theorem, which we recall here for the reader's convenience (and whose proof can be found in any advanced probability textbook).
\begin{thm}\label{OST}
	Let $(M_t)_{t\in \mathbb{N}_0}$ be a submartingale (resp.~ martingale, supermartingale) with respect to a filtration $(\mathcal{F}_t)_{t\in \mathbb{N}_0}$. Let $\sigma_1$ and $\sigma_2$ be bounded stopping times with respect to $(\mathcal{F}_t)_{t\in \mathbb{N}_0}$, satisfying $\sigma_1\leq \sigma_2$.  Then the random variables $M_{\sigma_1}$ and $M_{\sigma_2}$ are integrable and $\mathbb{E}[M_{\sigma_2}|\mathcal{F}_{\sigma_1}]\geq M_{\sigma_1}$ (resp. $\mathbb{E}[M_{\sigma_2}|\mathcal{F}_{\sigma_1}]= M_{\sigma_1}$, $\mathbb{E}[M_{\sigma_2}|\mathcal{F}_{\sigma_1}]\leq M_{\sigma_1}$).
\end{thm}


Throughout we assume that the random process  $(R_t)_{t\in [K]\cup \{0\}}$ satisfies the following recursion. We let $R_0=z\in \mathbb{N}$ and, for each $t\in [K]$:
\begin{itemize}
	\item $R_t=R_{t-1}+\eta_t-1$ if $R_{t-1}\geq 1$;
	\item $R_t=\eta_t$ if $R_{t-1}=0$,
\end{itemize}
for some non-negative random variable $\eta_t$. Let $N_1,N_2\in \mathbb{N}$ be such that $N_1\leq N_2<K$ and let $h>0$. (Later on we will choose $h$ to be significantly smaller than $\sqrt{N_2}$ but larger than $\sqrt{N_1}$). Following Nachmias and Peres \cite{nachmias_peres:CRG_mgs,nachmias:critical_perco_rand_regular} (and our discussion in the introductory section), our aim is to show that, under some basic first and second moment conditions involving the random variables $\eta_i$, then: (i) the process $R_t$ reaches level $h$ before time $N_2$; (ii) once having reached some state above $h$, it does not go back to zero in less than $N_1$ steps. 

To this end, we introduce the following stopping times.
\begin{defn}\label{defstop}
	Let $h>0$ and $N_1,N_2\in \mathbb{N}$ be such that $N_1\leq N_2<K$. We define:
	\begin{itemize}
		\item [(a)] $\tau'_h\coloneqq \min\{t\geq 1:R_t\geq h\}$ and $\tau_h\coloneqq \tau'_h\wedge N_2$; 
		\item [(b)] $\tau'_0\coloneqq \min\{t\geq 1:R_{\tau_h+t}=0\}$ and $\tau_0\coloneqq \tau'_0 \wedge N_1$.
	\end{itemize}
\end{defn}
Set $t_0\coloneqq 0$ and $t_i\coloneqq \inf\{t\geq t_{i-1}+1:R_t=0\}$ for $i\in \mathbb{N}$ (as long as we do not exceed the time window under consideration, i.e. prior to time $K$). 

Our aim is to provide an upper bound on the probability that all the positive excursions of $R_t$ last for less than $N_1$ steps, namely
\begin{equation}\label{rightform}
\mathbb{P}(t_i-t_{i-1}<N_1 \text{ }\forall i).
\end{equation}
According to our previous discussion, we bound from above the probability in (\ref{rightform}) by
\begin{equation}\label{THEsplit}
\mathbb{P}(\tau_h=N_2)+\mathbb{P}(t_i-t_{i-1}\leq N_1\text{ }\forall i, R_{\tau_h}\geq h).
\end{equation} 
Hence our goal is to estimate the two probabilities on the right-hand side of (\ref{THEsplit}), which we do by means of Propositions \ref{NEWmainprop} and \ref{secondmainprop} below, respectively.

In what follows, we work under the following two hypothesis (which we comment in Remark \ref{remonconds} below).
\begin{enumerate}
	\item [(H.1)] There exist a stopping time $\tau_1\in \mathbb{N}_{\infty}$ and constants $\varepsilon_1,\varepsilon_2\in (0,1),c_0>0,\sigma^2>1+3\varepsilon_2$ such that, for every $t\leq \tau_h\wedge \tau_1$:
	\begin{itemize}
		\item [(i)] $\mathbb{E}[\eta^2_{t}]\leq c_0$;
		\item [(ii)] $\mathbb{E}[\eta_t|R_{t-1}]\leq 1$ when $R_{t-1}\geq 1$;
		\item [(iii)] $1-\varepsilon_1\leq \mathbb{E}[\eta_t|R_{t-1}]$;
		\item [(iv)] $\mathbb{E}[\eta^2_t|R_{t-1}]\geq \sigma^2-\varepsilon_2$.
	\end{itemize}
	\item [(H.2)] There exist a stopping time $\tau_2\in \mathbb{N}_{\infty}$ and constants $c_1>0,\varepsilon_3\in (0,1)$ such that, for every $t\leq \tau_0\wedge \tau_2$:
	\begin{itemize}
		\item [(i)]  $\mathbb{E}[\eta^2_{\tau_h+t}|\eta_{\tau_h},\dots,\eta_{\tau_h+t-1},\tau_h]\leq c_1$;
		\item [(ii)]  $\mathbb{E}[\eta_{\tau_h+t}|\eta_{\tau_h},\dots,\eta_{\tau_h+t-1},\tau_h]\geq 1-\varepsilon_3 $ when $R_{\tau_h+t-1}<h$.
	\end{itemize}
\end{enumerate}
\begin{oss}\label{remonconds}
	Let us comment on the above assumptions (H.1) and (H.2). Recalling the description of an exploration process that we provided in the introductory section, it becomes clear that the (conditional) distribution of $\eta_t$ depends on $R_{t-1}$: the number of new active vertices (or edges) revealed at step $t$ depends on how many nodes are unseen (i.e. neither active nor explored) at the end of the previous step $t-1$. Sometimes, however, the (conditional) distributions of the $\eta_t$ also depend on some further random variables other than $R_{t-1}$. Thus, in order to compute (conditional) first and second moments of the $\eta_t$, we must have some control over this extra source of randomness; the stopping times $\tau_1$ and $\tau_2$ (whose precise definitions depend on the random graph under consideration) guarantee such control, thus allowing us to compute the necessary characteristics related to the $\eta_t$.
\end{oss}

\begin{oss}
	We remark that, probably, conditions (H.1) and (H.2) are not stated in the most general way; but, as we show in later sections, they are sufficient to include in the analysis all the random graphs described in the introductory section.
\end{oss}
\begin{oss}
	We remark that, in all our applications, the above conditions are easy to check; in particular, the upper bounds on the (conditional) first and second moments of the $\eta_t$ are very easy to show because these random variables are dominated by simpler distributions that are easy to analyse. We also remark that the term $h=h(n)$ satisfies $h\gg 1$ and each $\varepsilon_i=\varepsilon_i(n)$ is such that $\varepsilon_i\ll 1$ (for $i\in \{1,2,3\}$). Moreover, $\sigma^2$ and $z$ are of constant order (they do not depend on $n$).
\end{oss}
Under the above assumptions (H.1) and (H.2), we can bound (from above) the two probabilities on the right-hand side of (\ref{THEsplit}). Specifically, with the next proposition we show that the process $R_t$ is unlikely to remain below $h$ for $N_2$ consecutive steps, provided $h$ is significantly smaller than $\sqrt{N_2}$ (and $\mathbb{P}(\tau_1\leq N_2)$ is small too).
\begin{prop}\label{NEWmainprop}
	Under (H.1), if $c_0+1\leq h\leq \varepsilon^{-1}_1(\sigma^2-1)/6$ then
	\begin{equation}\label{boundofprop}
	\mathbb{P}(\tau_h=N_2)\leq \frac{3(\sigma^2-1)^{-1}}{N_2}(2h^2-z^2)+\mathbb{P}(\tau_1\leq N_2).
	\end{equation}
\end{prop}
\begin{proof}
	We proceed along the lines of \cite{nachmias_peres:CRG_mgs,nachmias:critical_perco_rand_regular}. Let $t\leq \tau_h\wedge \tau_1$. If $R_{t-1}\geq 1$ then, since $h\leq \varepsilon_1^{-1}(\sigma^2-1)/6$ and $\varepsilon_2\leq (\sigma^2-1)/3$, we can write
	\begin{align}\label{nowplug}
	\nonumber\mathbb{E}[R^2_{t-1}|R_{t-1}]&=R^2_{t-1}+2R_{t-1}(\mathbb{E}[\eta_t|R_{t-1}]-1)+\mathbb{E}[\eta^2_t|R_{t-1}]-2\mathbb{E}[\eta_t|R_{t-1}]+1\\
	\nonumber&\geq R^2_{t-1}+\sigma^2-1-2h\varepsilon_1-\varepsilon_2\\
	\nonumber&\ge R^2_{t-1}+2\frac{\sigma^2-1}{3}-\varepsilon_2\\
	&\geq R^2_{t-1}+\frac{\sigma^2-1}{3}.
	\end{align}
	Similarly, if $R_{t-1}=0$, since $\varepsilon_2<1$ (and $R^2_{t-1}=0$) we obtain 
	\begin{equation*}
	\mathbb{E}[R^2_{t-1}|R_{t-1}]\geq \sigma^2-\varepsilon_2>R^2_{t-1}+\sigma^2-1>R^2_{t-1}+(\sigma^2-1)/3.
	\end{equation*}
	Therefore the random process defined by $R^2_{t\wedge \tau_h\wedge \tau_1}-(\sigma^2-1)3^{-1}(t\wedge \tau_h\wedge \tau_1)$ is a submartingale and we can use the Optional Stopping Theorem applied to the bounded stopping times $0\eqqcolon\sigma_1\leq \sigma_2\coloneqq\tau_h\wedge \tau_1$ to conclude that 
	\begin{equation}\label{plug}
	\mathbb{E}[\tau_h\wedge \tau_1]\leq \frac{3}{\sigma^2-1}(\mathbb{E}[R^2_{\tau_h\wedge \tau_1}]-z^2).
	\end{equation}
	Now observe that, when $R_{(\tau_h\wedge \tau_1) -1}\geq 1$, we have
	\[R^2_{\tau_h\wedge \tau_1}\leq h^2+2h(\eta_{\tau_h\wedge \tau_1}-1)+\eta^2_{\tau_h\wedge \tau_1 }-2\eta_{\tau_h\wedge \tau_1}+1;\]
	since
	\begin{itemize}
		\item $\mathbb{E}[\eta_{\tau_h\wedge \tau_1}]=\mathbb{E}[\mathbb{E}(\eta_{\tau_h\wedge \tau_1}|R_{(\tau_h\wedge \tau_1)-1})]\leq 1$;
		\item $\mathbb{E}[\eta^2_{\tau_h\wedge \tau_1}]\leq c_0$;
		\item $\mathbb{E}[2\eta_{\tau_h\wedge \tau_1}]=2\mathbb{E}[\mathbb{E}(\eta_{\tau_h\wedge \tau_1}|R_{(\tau_h\wedge \tau_1)-1})]\geq 2(1-\varepsilon_1)$
	\end{itemize} 
	and $\varepsilon_1<1$, we obtain $\mathbb{E}[R^2_{\tau_h\wedge \tau_1 }]\leq h^2+c_0-1+2\varepsilon_1\leq h^2+c_0+1\le 2h^2$. If, on the other hand, $R_{(\tau_h\wedge \tau_1)-1}=0$, then $\mathbb{E}[R^2_{\tau_h\wedge\tau_1}]=\mathbb{E}[\eta^2_{\tau_h\wedge\tau_1}]\leq c_0<2h^2$ too. Plugging this estimate into (\ref{plug}) yields that 
	\[\mathbb{E}[\tau_h\wedge \tau_1]\leq \frac{3}{\sigma^2-1}(2h^2-z^2)\]
	and so, by Markov's inequality, we obtain
	\begin{equation}\label{afterm}
	\mathbb{P}(\tau_h\wedge \tau_1 =N_2)\leq \frac{\mathbb{E}[\tau_h\wedge \tau_1]}{N_2}\leq \frac{3(\sigma^2-1)^{-1}}{N_2}(2h^2-z^2).
	\end{equation}
	Therefore it follows from (\ref{afterm}) that
	\begin{align*}
	\mathbb{P}(\tau_h=N_2)&\leq \mathbb{P}(\tau_h=N_2,\tau_1>N_2)+\mathbb{P}(\tau_1\leq N_2)\\
	&\leq \mathbb{P}(\tau_h\wedge \tau_1 =N_2)+\mathbb{P}(\tau_1\leq N_2)\\
	&\leq \frac{3(\sigma^2-1)^{-1}}{N_2}(2h^2-z^2)+\mathbb{P}(\tau_1\leq N_2).\qedhere
	\end{align*}
\end{proof}
The next proposition (which again follows \cite{nachmias_peres:CRG_mgs,nachmias:critical_perco_rand_regular}) provides an upper bound on the second probability that appears on the right-hand side of (\ref{THEsplit}).
\begin{prop}\label{secondmainprop}
	Under (H.2), if $h\leq \varepsilon^{-1}_3$ then
	\begin{equation}\label{ratioo}
	\mathbb{P}(t_i-t_{i-1}<N_1 \text{ }\forall i,R_{\tau_h}\geq h)\leq \frac{(c_1+3)N_1}{h^2}+\frac{\mathbb{P}(\tau_2\leq N_1)}{\mathbb{P}(R_{\tau_h}\geq h)}.
	\end{equation}
\end{prop} 
\begin{oss}\label{twoeight}
	Note that $\mathbb{P}(R_{\tau_h}\geq h)\geq \mathbb{P}(\tau_h<N_2)=1-\mathbb{P}(\tau_h=N_2)$ and Proposition \ref{NEWmainprop} already gives us an upper bound on the probability that $\tau_h=N_2$ (provided (H.1) holds); this gives us a lower bound on $\mathbb{P}(R_{\tau_h}\geq h)$, which is what we need in order to estimate the ratio of probabilities on the right-hand side of (\ref{ratioo}).
\end{oss}
\begin{proof}
	Define $M_t\coloneqq (h-R_{\tau_h+t})\vee 0$. We claim that, if $0\leq M_{t-1}<h$, then
\begin{equation*}
M^2_{t}-M^2_{t-1}\leq (\eta_{\tau_h+t}-1)^2+2M_{t-1}(1-\eta_{\tau_h+t}).
\end{equation*}
Indeed, if $M_{t-1}=0$ then $R_{\tau_h+t-1}\geq h$ and so in particular
\begin{equation*}\label{impineqhere}
h-R_{\tau_h+t}=1-\eta_{\tau_h+t}-(R_{\tau_h+t-1}-h)\leq 1-\eta_{\tau_h+t},
\end{equation*}
so that $M_t\leq (1-\eta_{\tau_h+t})\vee 0$ and hence (as $M_{t-1}=0$) we obtain
\[M^2_t-M^2_{t-1}=M^2_t\leq (\eta_{\tau_h+t}-1)^2=(\eta_{\tau_h+t}-1)^2+2M_{t-1}(1-\eta_{\tau_h+t}).\]
On the other hand, when $1\leq M_{t-1}<h$ we have $h>R_{\tau_h+t-1}\geq 1$ and so expanding the squares we see that
\begin{align*}
M^2_t-M^2_{t-1}=(h-R_{\tau_h+t})^2-(h-R_{\tau_h+t-1})^2&= (\eta_{\tau_h+t}-1)^2+2(h-R_{\tau_h+t-1})(1-\eta_{\tau_h+t})\\
&=(\eta_{\tau_h+t}-1)^2+2M_{t-1}(1-\eta_{\tau_h+t})
\end{align*}
too, establishing the claim. Taking (conditional) expectations given $\eta_{\tau_h},\dots,\eta_{\tau_h+t-1}$ and $\tau_h$ on both sides of the last identity and using our assumption on $h$ we conclude that, if $t\leq \tau_0\wedge\tau_2$, then $\mathbb{E}[M^2_t-M^2_{t-1}|\eta_{\tau_h},\dots,\eta_{\tau_h+t-1},\tau_h]\leq c_1+3$. Thus the process defined by $M^2_{t\wedge \tau_0\wedge \tau_2}-(c_1+3)(t\wedge \tau_0\wedge\tau_2)$ is a supermartingale which, conditional on $R_{\tau_h}\geq h$, starts at zero. Therefore, by the Optional Stopping Theorem applied with the bounded stopping times $0\eqqcolon\sigma_1\leq \sigma_2\coloneqq\tau_0\wedge\tau_2$, we obtain
\[\mathbb{E}[M^2_{\tau_0\wedge\tau_2}|R_{\tau_h}\geq h]\leq (c_1+3)\mathbb{E}[\tau_0\wedge\tau_2]\leq (c_1+3)\mathbb{E}[\tau_0] \leq (c_1+3)N_1.\]
Next observe that, if $t_i-t_{i-1}<N_1$ for every $i$, then the process $R_{\tau_h+t}$ has to reach zero in less than $N_1$ steps. Therefore we obtain 
\begin{align}\label{sbh}
\nonumber\mathbb{P}(t_i-t_{i-1}<N_1\text{ }\forall i, R_{\tau_h}\geq h)&\leq \mathbb{P}(\tau_0<N_1|R_{\tau_h}\geq h)\\
\nonumber&\leq \mathbb{P}(M_{\tau_0}=h|R_{\tau_h}\geq h)\\
\nonumber&\leq \mathbb{P}(M_{\tau_0}=h,\tau_2>N_1|R_{\tau_h}\geq h)+\mathbb{P}(\tau_2\leq N_1|R_{\tau_h}\geq h)\\
&\leq  \mathbb{P}(M_{\tau_0\wedge \tau_2}=h|R_{\tau_h}\geq h)+\frac{\mathbb{P}(\tau_2\leq N_1)}{\mathbb{P}(R_{\tau_h}\geq h)},
\end{align}  
where the last inequality follows from the definition of conditional probability together with the observation that, since $\tau_0\leq N_1$ (by definition), if $\tau_2>N_1$ then automatically $\tau_0=\tau_0\wedge \tau_2$. Now since the first probability on the right-hand side of (\ref{sbh}) is at most
\begin{equation*}
\mathbb{P}(M^2_{\tau_0\wedge \tau_2}\geq h^2|R_{\tau_h}\geq h)\leq \frac{\mathbb{E}[M^2_{\tau_0\wedge\tau_2}|R_{\tau_h}\geq h]}{h^2}\leq (c_1+3)\frac{N_1}{h^2},
\end{equation*}
we conclude that
\begin{equation*}
\mathbb{P}(t_i-t_{i-1}<N_1\text{ }\forall i, \tau_h<N_2)\leq (c_1+3)\frac{N_1}{h^2}+\frac{\mathbb{P}(\tau_2\leq N_1)}{\mathbb{P}(R_{\tau_h}\geq h)}.\qedhere
\end{equation*}
\end{proof}
Under (H.1) and (H.2), Propositions \ref{NEWmainprop} and \ref{secondmainprop} together with (\ref{THEsplit}) (and Remark \ref{twoeight}) give us an upper bound on the probability that all positive excursions of the process $R_t$ last for less than $N_1$ steps (that is, the probability in (\ref{rightform})).
\begin{prop}\label{jointprop}
	Let $(R_t)_{t\in [K]\cup\{0\}}$ be as above. Then, under (H.1) and (H.2), if $c_0\vee 1\leq h\leq \big(\frac{\sigma^2-1}{6}\varepsilon^{-1}_1\big) \wedge \varepsilon^{-1}_3$, setting 
	\[\Phi=\Phi(N_1,N_2,h,\sigma^2,z)\coloneqq \frac{3(\sigma^2-1)^{-1}}{N_2}(2h^2-z^2)+\mathbb{P}(\tau_1\leq N_2)\]
	we obtain 
	\begin{equation*}
	\mathbb{P}(t_i-t_{i-1}<N_1 \text{ }\forall i)\leq \Phi+\frac{(c_1+3)N_1}{h^2}+\frac{\mathbb{P}(\tau_2\leq N_1)}{1-\Phi}.
	\end{equation*}
\end{prop}
\begin{oss}
	We remark that, in a previous version of this work, a different argument was used to bound (from above) the probability in (\ref{boundofprop}). Specifically, a strong embedding argument of Brownian motion and random walk \cite{chatterjee:strong_embeddings} was used to show that, if $S_t=\sum_{i=1}^{t}X_i$ and the $X_i$ are iid $\mathbb{Z}$-valued random variables with $\mathbb{E}[X_1]=0<\mathbb{E}[X^2]<\infty$ and $\mathbb{E}[e^{\rho|X_1|}]<\infty$ for some $\rho>0$, then with probability at least $1-O((h\vee \log(N_2))/\sqrt{N_2})$ the random walk $S_t$ crosses the barrier $h$ at some time $t<N_2$. Based on this result, the main task in order to obtain an upper bound for the probability on the left-hand side of (\ref{boundofprop}) was then to approximate the random process $R_t$ with a random walk satisfying the above conditions. Such an approximation was achieved by means of simple couplings and change of measure arguments (needed to remove the drift from the process $R_t$ as well as the dependence of $n$ from its increments $\eta_i$). Even though such a method was quite general, we believe that Proposition \ref{NEWmainprop} is nicer because it can be applied in contexts where removing the dependence of $n$ from the law of the $\eta_i$ could be difficult due to the nature of the model at hand. 
\end{oss}
We conclude this subsection by recalling a standard concentration inequality for the binomial distribution.
\begin{lem}\label{cher}
	Let $X$ be a random variable with the $\text{Bin}(N,q)$ distribution, where $N\in \mathbb{N}$ and $q\in [0,1]$. Then, for any $x>0$, we have
	\[\mathbb{P}(X\geq Nq+x)\leq \exp\left\{-\frac{x^2}{2Nq+(2x)/3}\right\}.\]
\end{lem}

\subsection{Preliminary result for Theorem \ref{corollary}: a ballot-type estimate}\label{subuppertail}
The proof of Theorem \ref{corollary} relies on Lemmas \ref{simplelem} and \ref{secondsimplelem} below; the former lemma is solely based on the following \textit{ballot-type} result, introduced in \cite{de_ambroggio_roberts:near_critical_ER} (see also \cite{de_ambroggio:component_sizes_crit_RGs} and \cite{de_ambroggio_roberts:near_critical_RRG}) in order to establish a sharp upper bound for the probability of observing unusually \textit{large} maximal components in the near-critical $\mathbb{G}(n,p)$ random graph.

The ballot theorem concerns the probability that a random walk $S_t=\sum_{i=1}^{t}X_i$ (with iid, mean-zero increments with finite second moment) stays positive for all times $t\in [n]$, given that $S_n=k \in \mathbb{N}$; see e.g. \cite{addario_berry_reed:ballot_theorems,de_ambroggio_roberts:near_critical_ER} and references therein. 

The following result gives us an upper bound for the probability that $S_t$ stays positive for $n$ steps and finishes (at time $n$) at level $j\in \mathbb{N}$.
\begin{lem}[Corollary 2.4 in \cite{de_ambroggio_roberts:near_critical_ER}]\label{lemmaballot}
	Fix $n\in\N$ and let $(X_i)_{i\in \mathbb{N}}$ be iid random variables taking values in $\mathbb{Z}$, whose distribution may depend on $n$. Let $r\in \mathbb{N}$ and suppose that $\mathbb{P}(X_1=r)>0$. Define $S_t = \sum_{i=1}^{t}X_i$ for $t\in\mathbb{N}_0$. Then, for every $a\in [0,\infty)$ (possibly dependent on $n$) and for any $j\in \mathbb{N}$, we have
	\begin{equation}\label{probwebound}
	\mathbb{P}(r+S_t>a\hspace{0.15cm} \forall t\in [n],\, r+S_{n}=j)\leq \mathbb{P}(X_1=r)^{-1}\frac{j}{(n+1)\lfloor a+1\rfloor}\mathbb{P}(S_{n+1}=j).
	\end{equation}
\end{lem}
\begin{oss}
	We remark that, in \cite{de_ambroggio_roberts:near_critical_ER}, this result was only stated for the case $a=0$; but the proof it's exactly the same \textit{for every} $a\geq 0$ (which in particular it is allowed to depend on $n$).
\end{oss} 
In order to make the paper self-contained, we sketch the proof of Lemma \ref{lemmaballot} (which is short and simple).

Fix $n\in \mathbb{N}$ and let $X_1,\dots,X_n$ be iid random variables taking values in $\mathbb{Z}$. Define $S_0 = 0$ and $S_t = \sum_{i=1}^{t}X_i$ for all $t\in [n]$. Let $r\in [n]$ and define $S^r_t$ as the sum of the first $t$ random variables $X_i$ starting from $X_{r+1}$; formally, we set:
\begin{itemize}
	\item $S_t^{r} \coloneqq S_{t+r}-S_r = \sum_{i=r+1}^{t+r}X_i$ if $0\leq t\leq n-r$;
	\item $S_t^{r} \coloneqq S_n+S_{t+r-n}-S_r = \sum_{i=r+1}^{n}X_i+\sum_{i=1}^{t+r-n}X_i$ if if $n-r<t\leq n$.
\end{itemize}
We call $S^{r}=(S_0^{r},S_1^{r},\dots,S_n^{r})$ the \textit{rotation} of $S=(S_0,S_1,\dots,S_n)$ by $r$. Note that $S_n^{r}=\sum_{i=r+1}^{n}X_i+\sum_{i=1}^{r}X_i=S_n$ (for every $r\in [n]$) and $S^n=S$. For $r\in [n]$, we say that
\[r \text{ is \textbf{favourable} if } S_t^{r}>a \text{ for every } t\in [n].\]
The key fact that we need in order to prove Lemma \ref{lemmaballot} is (\ref{Claim}) below. Namely we claim that, for any $j\in \mathbb{N}$, if $S_n=j$ then
\begin{equation}\label{Claim}
|\{r\in[n]: r \text{ is favourable}\}|\leq j/\lfloor a+1\rfloor.
\end{equation}
To see this, denoting by $1\leq I_1<\dots <I_L\leq n$ the \textit{favourable indices} (if any), we observe that for each $1\leq k\leq L-1$ (since the $X_i$ are integer-valued) we have $S^{I_k}_t>a$, which implies $S^{I_k}_t\geq \lfloor a+1\rfloor$ for all $t\in[n]$; in particular $S_{I_{k+1}-I_k}^{I_k}\geq \lfloor a+1\rfloor$. Analogously, $S_{(I_1+n)-I_L}^{I_L}\geq \lfloor a+1\rfloor$, whence
	\[j=S_n= \sum_{k=1}^{L-1}S_{I_{k+1}-I_k}^{I_k}+S_{(I_1+n)-I_L}^{I_L} \ge (L-1)\lfloor a+1\rfloor +\lfloor a+1\rfloor = L\lfloor a+1\rfloor,\]
	establishing the claim (\ref{Claim}). To prove Lemma \ref{lemmaballot} we first note that, for any $r\in[n]$, since the $X_i$ are independent we can write 
	\begin{equation*}
	\mathbb{P}(S_t>0\hspace{0.15cm}\forall t\in [n],\,S_n=j) = \mathbb{P}(S_t^r>0\hspace{0.15cm}\forall t\in [n],\, S_n^r=j)=\mathbb{P}(r\text{ is favourable},\,S_n=j),
	\end{equation*}
	where the last identity follows from the fact that $S_n^r=S_n$.
	Summing over $r\in[n]$ we obtain
	\begin{equation}\label{kkey}
	 \mathbb{P}(S_t>0\text{ }\forall t\in [n],\,S_n=j)  = \frac{1}{n}\mathbb{E}\bigg[\ind_{\{S_n=j\}}\sum_{r=1}^{n}\ind_{\{r \text{ is favourable}\}}\bigg]
	\leq \frac{j}{\lfloor a+1\rfloor n}\mathbb{P}(S_n=j).
	\end{equation}
	(We remark that, up to this point, we have only used the independence of the $X_i$.) Note that, compared to the probability which appears in the statement of the lemma, we are still missing the term $r$ in the probability on the left-hand side of (\ref{kkey}). To fix this, let $X_0$ be an independent copy of $X_1$. Define $S_t^*=X_0+S_t$ for $0\leq t\leq n$. Then the probability on the left-hand side of (\ref{probwebound}) equals
	\begin{multline}
	\nonumber\mathbb{P}(X_0=h)^{-1}\mathbb{P}(h+S_t>a\text{ } \forall t\in [n],\,h+S_{n}=j,\,X_0=h)\\
	\leq \mathbb{P}(X_1=h)^{-1}\mathbb{P}(S_t^*>a\text{ } \forall t\in \{0\}\cup [n],\,S_{n}^*=j).
	\end{multline}
	But $(S_0^*,S_1^*,\dots,S_n^*)\overset{d}{=}(S_1,S_2,\dots,S_{n+1})$ and so we can use (\ref{kkey}) to bound from above the expression on the right-hand side of the last inequality by
	\begin{equation*}
	\mathbb{P}(X_1=h)^{-1}\mathbb{P}(S_t>a\text{ } \forall t\in [n+1]\,,S_{n+1}=j)\leq \mathbb{P}(X_1=h)^{-1}\frac{j}{(n+1)\lfloor a+1\rfloor}\mathbb{P}(S_{n+1}=j),
	\end{equation*}
	establishing Lemma \ref{lemmaballot}.

Our next goal is to use Lemma \ref{lemmaballot} to provide a very short and simple proof of the well-know fact \cite{rwbook} that a random walk (with iid increments having zero mean and finite second moment) reaches $0$ in less than $k$ steps with probability at least $1-O(k^{-1/2})$; Theorem \ref{corollary} is then a direct consequence of this fact and Lemma \ref{secondsimplelem} below.

We state the lemma in a form that is more convenient for our applications, namely for a random walk with increments $X_i-1$ where $X_i$ are iid and take values in $\mathbb{N}_0$.
\begin{lem}\label{simplelem}
	Let $(X_i)_{i\in \mathbb{N}}$ be a sequence of iid random variables taking values in $\mathbb{N}_0$ such that $\mathbb{E}[X_1]\leq 1$ and $\mathbb{V}[X_1]<\infty$. Let $r\in \mathbb{N}$ and set $S_t\coloneqq \sum_{i=1}^{t}(X_i-1)$ for $t\in \mathbb{N}$. Suppose that $c_r\coloneqq\mathbb{P}(X_1=r+1)>0$. Then, for every $k\in \mathbb{N}$, we have
	\begin{equation}\label{goal}
	\mathbb{P}(r+S_t>0\text{ }\forall t\leq k)\leq \frac{c_r(1+2\mathbb{V}[X_1])}{k^{1/2}}.
	\end{equation}
\end{lem}
\begin{proof}
	Let $m<k$ be a positive integer to be specified and observe that, thanks to Lemma \ref{lemmaballot}, we can bound from above the probability in (\ref{goal}) by
	\begin{multline}\label{kjj}
	\frac{c_r m}{k+1}+c_r(k+1)^{-1}\sum_{j\geq m+1}^{}j\mathbb{P}(S_{k+1}=j)\\
	=\frac{c_r m}{k+1}+c_r(k+1)^{-1}\sum_{\ell \geq k+m+2}^{}[\ell -(k+1)]\mathbb{P}(\sum_{i=1}^{k+1}X_i=\ell),
	\end{multline}
	where we have used the change of variables $\ell\coloneqq k+1+j$ to obtain the last equality. Now, for any non-negative random variable $X$ and any $h\in \mathbb{N}$ we have $\mathbb{E}[X\mathbb{1}{X\geq h}]=h\mathbb{P}(X\geq h)+\sum_{j>h}^{}\mathbb{P}(X\geq j)$; using such an identity for the (non-negative) random variable $\sum_{i=1}^{k+1}X_i$, a short computation reveals that the second term on the right-hand side of (\ref{kjj}) equals
	\begin{equation}\label{here}
	c_r\frac{m+1}{k+1}\mathbb{P}(\sum_{i=1}^{k+1}(X_i-1)\geq m+1)
	+c_r(k+1)^{-1}\sum_{j\geq m+2}^{}\mathbb{P}(\sum_{i=1}^{k+1}(X_i-1)\geq j).
	\end{equation}
	Since the random variables $X_i$ have mean at most one and are iid, we can use Chebyshev's inequality to conclude that $\sum_{i=1}^{k+1}(X_i-1)\geq j$ with probability at most $(k+1)\mathbb{V}[X_1]j^{-2}$ for all $j\geq m+1$, so that the expression in (\ref{here}) is at most
	\begin{equation*}
	c_r\mathbb{V}[X_1]\left((m+1)^{-1}+\int_{m+1}^{\infty}x^{-2}dx\right)=2c_r\mathbb{V}[X_1](m+1)^{-1}.
	\end{equation*}
	Plugging this estimate into (\ref{here}) and going back to (\ref{kjj}) we obtain
	\begin{equation*}
	\frac{c_r m}{k+1}+c_r(k+1)^{-1}\sum_{j\geq m+1}^{}j\mathbb{P}(S_{k+1}=j)\leq \frac{c_rm}{k+1}+\frac{2c_r\mathbb{V}[X_1]}{m+1};
	\end{equation*}
	setting $m\coloneqq \lfloor (k+1)^{1/2}\rfloor$ we obtain $m+1\geq (k+1)^{1/2}$ and the desired result follows. 
\end{proof}
We now use Lemma \ref{simplelem} to prove the following result, which can be seen as a better version of Theorem 1 in \cite{de_ambroggio:component_sizes_crit_RGs}.
\begin{lem}\label{secondsimplelem}
	Let $V_n$ be a node selected uniformly at random from $[n]$, the vertex set of a random graph $\mathbb{G}$. Suppose that there exist iid random variables $X_i$ with $\mathbb{E}[X_1]\leq 1, \mathbb{E}[X^2_1]<\infty$ and $\mathbb{P}(X_1=r+1)>0$ ($r\in \mathbb{N}$) such that, setting $S_t\coloneqq \sum_{i=1}^{t}(X_i-1)$, we can bound from above the probability that $|\mathcal{C}(V_n)|>k\in \mathbb{N}$ by $c_0\mathbb{P}(r+S_t>0 \text{ }\forall t\leq k)$, for some $c_0>0$. Then, setting $c_*\coloneqq c_0\mathbb{P}(X_1=r+1)^{-1}(1+2\mathbb{V}[X_1])$, we have
	\[\mathbb{P}(|\mathcal{C}_{\max}(\mathbb{G})|>k)\leq c_*\frac{n}{k^{3/2}}.\]
\end{lem}
\begin{proof}
	We make use of a standard trick (see e.g. \cite{nachmias_peres:CRG_mgs,nachmias:critical_perco_rand_regular} and \cite{hofstad_critic}, \cite{de_ambroggio_roberts:near_critical_ER}) to derive an upper bound on the probability that $\mathcal{C}_{\max}(\mathbb{G})$ contains more than $k$ vertices in terms of an upper bound on the probability that $|\mathcal{C}(V_n)|>k$. For $k\in \mathbb{N}$, define $Z_{>k}\coloneqq \sum_{i=1}^{n}\mathbb{1}\{|\mathcal{C}(i)|>k\}$, the number of nodes contained in components formed by more than $k$ vertices. Then, by Markov's inequality, we see that
	\begin{align}\label{nowwimplel}
	\nonumber\mathbb{P}(|\mathcal{C}_{\max}(\mathbb{G})|>k)\leq \mathbb{P}(Z_{> k}\geq k)\leq \frac{\mathbb{E}[Z_{>k}]}{k}&=k^{-1}n\mathbb{P}(|\mathcal{C}(V_n)|>k)\\
	&\leq k^{-1}nc_0\mathbb{P}(r+S_t>0\text{ }\forall t\leq k).
	\end{align} 
	Since $\mathbb{P}(X_1=r+1)>0$, we can apply Lemma \ref{simplelem} to conclude that the probability on the right-hand side of (\ref{nowwimplel}) is at most $k^{-3/2}nc_0\mathbb{P}(X_1=2)^{-1}(1+2\mathbb{V}[X_1])$, which is the desired result.
\end{proof}

\section{Proof of Theorem \ref{mainthm3}}
The goal here is to use Proposition \ref{jointprop} to show that the random graphs described in the introductory section, when considered at criticality of their parameters, exhibit largest components containing more than $n^{2/3}/A$ nodes with probability at least $1-cA^{-1/2}$, where $c>0$ is a constant which depends on the model under consideration. 

\subsection{The $\mathbb{G}(n,p)$ model}\label{subsectionGnp}
We want to show that, in the critical $\mathbb{G}(n,p)$ model (which we recall it is the random graph obtained by performing $p$-bond percolation on the complete graph on $n$ vertices when $p=1/n$), it is unlikely to observe a largest component containing significantly less than $n^{2/3}$ vertices.
We start by introducing an exploration process to reveal the connected components of $\mathbb{G}(n,p)$; the description that we provide here is taken from \cite{de_ambroggio_roberts:near_critical_ER}, which is based on \cite{nachmias_peres:CRG_mgs}.

\subsubsection{An exploration process}\label{explorationGnp}
During the procedure, each vertex is \textit{active}, \textit{explored} or \textit{unseen}. At time $t=0$, we chose a node $v\in [n]$ and declare it active; all other vertices in $[n]$ are declared unseen. At time $t=1$, we reveal all the (necessarily) unseen vertices in $[n]$ which are connected to $v$ and change the status of these nodes to active, whereas $v$ becomes explored. At time $t=2$, we let $v'$ be one of the active vertices, if there is at least one such node; otherwise, we select $v'$ from the set of unseen vertices. Then we reveal all the unseen neighbours of $v'$ and declare these nodes active, whereas $v'$ becomes explored. The procedure continues until all vertices in $[n]$ are in status explored. 

We now provide the formal description. Fix and ordering of the $n$ vertices with $v$ first. Let us denote by $\mathcal{A}_t$,  $\mathcal{U}_t$ and $\mathcal{E}_t$ the (random) sets of active, unseen and explored vertices at the end of step $t\in \mathbb{N}_0$, respectively. Then, for every $t\in \mathbb{N}_0$, we have $[n]=\mathcal{A}_t\cup \mathcal{U}_t\cup \mathcal{E}_t$ (a disjoint union), so that in particular $\mathcal{U}_t=[n]\setminus (\mathcal{A}_t\cup \mathcal{E}_t)$ at each step $t$.\\ 

\textbf{Generate $\mathbb{G}(n,p)$}. At time $t=0$, vertex $v$ is declared active whereas all other vertices are declared unseen, so that $\mathcal{A}_0=\{v\}$, $\mathcal{U}_0=[n]\setminus\{v\}$ and $\mathcal{E}_0=\emptyset$. For every $t\in \mathbb{N}$, the algorithm proceeds as follows.
\begin{itemize}
	\item [(a)] If $|\mathcal{A}_{t-1}|\geq 1$, we let $u_t$ be the first active vertex (here and in what follows, the term \textit{first} refers to the ordering that we have fixed at the beginning of the procedure).
	\item [(b)] If $|\mathcal{A}_{t-1}|=0$ and $|\mathcal{U}_{t-1}|\geq 1$, we let $u_t$ be the first unseen vertex.
	\item [(c)] If $|\mathcal{A}_{t-1}|=0=|\mathcal{U}_{t-1}|$ (so that $\mathcal{E}_{t-1}=[n]$), we halt the procedure.
\end{itemize}
Denote by $\mathcal{D}_t$ the (random) set of \textit{unseen} neighbours of $u_t$, i.e. $\mathcal{D}_t\coloneqq \left\{x\in \mathcal{U}_{t-1}\setminus\{u_t\}:u_t\sim x \right\}$ (and note that $\mathcal{U}_{t-1}\setminus\{u_t\}=\mathcal{U}_{t-1}$ if $\mathcal{A}_{t-1}\neq \emptyset$). Then we update
\begin{align*}
\mathcal{U}_t\coloneqq  
\begin{cases}
\mathcal{U}_{t-1}\setminus \mathcal{D}_t,& \text{ if $|\mathcal{A}_{t-1}|\geq 1$}\\ 
\mathcal{U}_{t-1}\setminus (\mathcal{D}_t\cup \{u_t\}), & \text{ if $|\mathcal{A}_{t-1}|=0$}
\end{cases}
\end{align*} 
and $\mathcal{A}_t\coloneqq (\mathcal{A}_{t-1}\setminus\{u_t\})\cup \mathcal{D}_t$, $\mathcal{E}_t\coloneqq \mathcal{E}_{t-1}\cup \{u_t\}$.\\

\begin{oss}
	Note that, since in the procedure \textbf{Generate $\mathbb{G}(n,p)$} we explore \textit{one vertex} at each step, we have $\mathcal{A}_t\cup \mathcal{U}_t\neq \emptyset$ for every $t\leq n-1$ and $\mathcal{A}_n\cup \mathcal{U}_n=\emptyset$ (as $\mathcal{E}_n=[n]$). Thus the algorithm runs for $n$ steps.
\end{oss}
Denoting by $\eta_t$ the (random) number of unseen vertices that we add to the set of active nodes at time $t$, since at the end of each step $i$ in which $|\mathcal{A}_{i-1}|\geq 1$ we remove the (active) vertex $u_i$ from $\mathcal{A}_{i-1}$ (after having revealed its unseen neighbours), we have the recursion
\begin{itemize}
	\item $|\mathcal{A}_t|=|\mathcal{A}_{t-1}|+\eta_t-1$, if $|\mathcal{A}_{t-1}|>0$;
	\item $|\mathcal{A}_{t}|=\eta_t$, if $|\mathcal{A}_{t-1}|=0$.
\end{itemize}
Let us denote by $R_t$ and $U_t$ the number of active and unseen vertices at the end of step $t\in [n]\cup \{0\}$, respectively, so that
\[R_t= |\mathcal{A}_t| \text{ and }U_t= |\mathcal{U}_t|=n-t-R_t.\]
Note that, conditional on $R_{t-1}$, the random variable $\eta_t$ depends on the past until time $t-1$ only through $R_{t-1}$ and, in particular, we have
\begin{equation}\label{lawetagnp}
\eta_t=_d\Bin(U_{t-1}-\mathbb{1}_{\{R_{t-1}=0\}},1/n) \text{ for }t\in [n].
\end{equation}

\subsubsection{Upper bound on $\mathbb{P}(|\mathcal{C}_{\max}(\mathbb{G}(n,p))|<n^{2/3}/A)$}\label{upperGnp}
Define $t_0\coloneqq 0$ and $t_i\coloneqq \min\{t\geq t_{i-1}+1:R_t=0\}$, for $i\geq 1$ (as long as we do not enter Step $(c)$ in the procedure \textbf{Generate $\mathbb{G}(n,p)$}). Then, denoting by $\mathcal{C}_i$ the $i$-th component revealed during the exploration process, we have $|\CC_i|=t_i-t_{i-1}$ for every $i$ and so, setting $T\coloneqq \lceil n^{2/3}/A\rceil$, we can write
\begin{equation}\label{excgnp}
\mathbb{P}(|\mathcal{C}_{\max}(\mathbb{G}(n,p))|<n^{2/3}/A)\leq \mathbb{P}(|\mathcal{C}_{\max}(\mathbb{G}(n,p))|<T)=\mathbb{P}(t_i-t_{i-1}<T \text{ }\forall i).
\end{equation}
Note that the probability on the right-hand side of (\ref{excgnp}) is in the form we want, meaning that it equals (\ref{rightform}) (upon taking $N_1=T$). Therefore, all we need to do in order to obtain an upper bound for such a quantity is to check that conditions (H.1) and (H.2) are satisfied in this setting. 

Throughout the rest of this section we let $p=1/n,h>0$ and $T'\in \mathbb{N}$ be such that $T\leq T'=T'(n)\ll n$. We also assume that $h\leq T'$. (Later on we will take $h$ to be significantly smaller than $\sqrt{T'}$ but larger than $\sqrt{T}$.) 

Moreover, we let $\tau_h$ and $\tau_0$ be as in definition \ref{defstop} with $T'$ and $T$ in place of $N_2$ and $N_1$, respectively. Furthermore, we take $\tau_1=\infty=\tau_2$ (so that trivially $\tau_h\wedge\tau_1=\tau_h$ and $\tau_0\wedge\tau_2=\tau_0$). 

With the next two lemmas we show that conditions (H.1) and (H.2) are fulfilled. We start by checking (H.1).
\begin{lem}\label{firstbitER}
	For every $t\leq \tau_h$ and for all large enough $n$: 
	\begin{itemize}
		\item $\mathbb{E}[\eta^2_{t}]\leq 2$;
		\item $\mathbb{E}[\eta_t|R_{t-1}]\leq 1$; 
		\item $1-(4T')/n\leq \mathbb{E}[\eta_t|R_{t-1}]$;
		\item $\mathbb{E}[\eta^2_t|R_{t-1}]\geq 2-(5T')/n$.
	\end{itemize}
\end{lem}
\begin{proof}
	Note that, for every $t\leq \tau_h$, we can bound $\mathbb{E}[\eta^2_t]\leq \mathbb{E}[\text{Bin}^2(n,p)]\leq 2$. Moreover, since (by definition) $\tau_h\leq T'$ and (by hypothesis) $1\leq h\leq T'$, recalling (\ref{lawetagnp}) it is easy to show that $\mathbb{E}[\eta_t|R_{t-1}]\geq 1-4T'/n$ and also $\mathbb{E}[\eta^2_t|R_{t-1}]\geq 2-5T'/n$. Clearly $\mathbb{E}[\eta_t|R_{t-1}]\leq 1$ always.
\end{proof}
We proceed by checking that also (H.2) is satisfied in the current setting.
\begin{lem}\label{secondbitER}
	For every $t\leq \tau_0$ and for all large enough $n$:
	\begin{itemize}
		\item  $\mathbb{E}[\eta^2_{\tau_h+t}|\eta_{\tau_h},\dots,\eta_{\tau_h+t-1},\tau_h]\leq 2$;
		\item   $\mathbb{E}[\eta_{\tau_h+t}|\eta_{\tau_h},\dots,\eta_{\tau_h+t-1},\tau_h]\geq 1- (8dT')/n$ when $R_{\tau_h+t-1}<h$.
	\end{itemize}
\end{lem}
\begin{proof}
	The first statement is immediate as for Lemma \ref{firstbitER} and the second statement follows again from (\ref{lawetagnp}) after a simple computation.
\end{proof}
We are now in the position to show that $\mathcal{C}_{\max}(\mathbb{G}(n,p))$ contains more than $n^{2/3}/A$ vertices with probability at least $1-O(A^{-1/2})$. Indeed, let $T'\coloneqq \lceil n^{2/3}\rceil\geq T$. Then it follows from Proposition \ref{jointprop} together with (\ref{excgnp}) that, setting $c_0=2,\varepsilon_1=(4T')/n,\sigma^2=2, \varepsilon_2=(5T')/n,c_1=2,\varepsilon_3=(8T')/n$ and taking $h\coloneqq n^{1/3}/(24A^{1/4})$, 
\begin{equation*}
\mathbb{P}(|\CC_{\max}(\mathbb{G}(n,p))|<n^{2/3}/A)\leq \frac{C_1}{A^{1/2}}
\end{equation*}
for all sufficiently large $n$, for some constant $C_1>0$.

\subsection{The $\mathbb{G}(n,d,p)$ model}\label{subsectionGndp}
In order to derive our result for the $\mathbb{G}(n,d,p)$ model, which we recall is the random graph obtained through $p$-bond percolation on a (simple) $d$-regular graph selected uniformly at random from the set of all simple $d$-regular graphs on $[n]$, we analyse the percolated version of a random $d$-regular multigraph by means of the \textit{configuration model} \cite{bollobas_config}. In particular we will show that, for such a random $d$-regular multigraph, $\mathbb{P}(|\CC_{\max}|<n^{2/3}/A)$ decays as $A^{-1/2}$. Before explaining how to relate this fact to the desired bound for the random graph $\mathbb{G}(n,d,p)$, let us first recall the basic description of the configuration model (see also \cite{remco:random_graphs}).\\

\textbf{Configuration model}. Start with $dn$ stubs (or half-edges), labelled $(v,i)$ for $v\in[n]$ and $i\in[d]$. Choose a stub $(v_0,i_0)$ in some way (the manner of choosing may be deterministic or random) and pair it with another half-edge $(w_0,j_0)$ selected uniformly at random. Say that these two stubs are \emph{matched} (or paired) and put $v_0w_0\in E$. Then, at each subsequent step $k\in\{1,\ldots, nd/2-1\}$ (recall that $dn$ is even), choose an half-edge $(v_k,i_k)$ in some way from the set of unmatched stubs and pair it with another half-edge $(w_k,j_k)$ selected uniformly at random from the set of all unpaired stubs. Say that these two stubs are matched and put $v_kw_k\in E$. At the end of this procedure, we obtain a random \emph{multigraph} $\mathbb{G}'(n,d)$ but, with probability converging to $c_d\coloneqq \exp\{(1-d^2)/4\}$, $\mathbb{G}'(n,d)$ is a simple graph and, conditioning on this event, it is \textit{uniformly} chosen amongst all $d$-regular (simple) graphs on the vertex set $[n]$.\\

To obtain our result for the $\mathbb{G}(n,d,p)$ model we argue in the following standard way (see e.g. \cite{de_ambroggio_roberts:near_critical_RRG},\cite{nachmias:critical_perco_rand_regular}). Writing $\mathbb{S}_n$ for the event that the multigraph $\mathbb{G}'(n,d)$ generated by means of the configuration model is simple, if $n$ is large enough we have $\mathbb{P}(\mathbb{S}_n) \geq  c_d/2$. Moreover, since the (conditional) law of $\mathbb{G}'(n,d)$ given $\mathbb{S}_n$ coincides with that of $\mathbb{G}(n,d)$ (that is, the graph selected uniformly at random from the set of all simple $d$-regular graphs with $n$ vertices on which we perform $p$-bond percolation to obtain $\mathbb{G}(n,d,p)$), denoting by $\mathbb{G}'(p)$ the $p$-percolated version of $\mathbb{G}'(n,d)$ we obtain (for any $T\in \mathbb{N}$)
\begin{multline}\label{switchgraph}
\mathbb{P}\left(|\CC_{\max}(\mathbb{G}(n,d,p))|<T\right)=\mathbb{P}\left(|\CC_{\max}(\mathbb{G}'(p))|<T|\mathbb{S}_n\right)\\
\leq \frac{\mathbb{P}\left(|\CC_{\max}(\mathbb{G}'(p))|<T\right)}{\mathbb{P}(\mathbb{S}_n)}\leq \frac{2}{c_d}\mathbb{P}\left(|\CC_{\max}(\mathbb{G}'(p))|<T\right),
\end{multline}
provided $n$ is large enough. Therefore, we can deduce our result for the $\mathbb{G}(n,d,p)$ model by studying the random graph $\mathbb{G}'(p)$.  


\subsubsection{An exploration process}\label{explorationGndp}
The exploration process that we employ here, which is taken from \cite{de_ambroggio_roberts:near_critical_RRG} (which in turn it is inspired from the description in \cite{nachmias:critical_perco_rand_regular}), uses the configuration model to generate components of $\mathbb{G}'(p)$, the $p$-percolated version of a uniformly random $d$-regular multigraph $\mathbb{G}'(n,d)$. 

During the exploration process, each stub of the (random) $d$-regular multigraph $\mathbb{G}'(n,d)$ is either \textit{active}, \textit{unseen} or \textit{explored}, and its status changes during the course of the procedure. We denote by $\mathcal{A}_{t}$, $\mathcal{U}_{t}$ and $\mathcal{E}_{t}$ the sets of active, unseen and explored half-edges at the end of the $t$-th step of the exploration process, respectively.

Given a stub $h$ of $\mathbb{G}'(n,d)$, we denote by $v(h)$ the vertex incident to $h$ (in other words, if $h=(u,i)$ for some $i\in [d]$ then $v(h) = u$) and we write $\mathcal{S}(h)$ for the set of \textit{all} half-edges incident to $v(h)$ in $\mathbb{G}'(n,d)$ (that is, $\mathcal{S}(h) = \{(v(h),i) : i\in[d]\}$; note in particular that $h\in \mathcal{S}(h)$).

Let $V_n$ be a node selected uniformly at random from the vertex set $[n]$.  We let $(J_t)_{t\in \mathbb{N}}$ be a sequence of iid $\text{Ber}(p)$ random variables, independent of the random pairing which we discuss next; we use the $J_t$ to decide whether to keep the edge created at step $t$.\\

\textbf{Generate $\mathbb{G}'(p)$}. At step $t=0$ we declare \textit{active} all half-edges incident to $V_n$, while all the other $d(n-1)$ stubs are declared \textit{unseen}. Since there are $d$ half-edges incident to $V_n$ we have that $|\mathcal{A}_0|=d$, $|\mathcal{U}_0|=d(n-1)$ and $|\mathcal{E}_0|=0$. For every $t\in \mathbb{N}$, the algorithm proceeds as follows.
\begin{itemize}
	\item [(a)] If $|\mathcal{A}_{t-1}|\geq 1$, we choose (in an arbitrary way) one of the active stubs, say $e_t$, and we pair it with an half-edge $h_t$ selected uniformly at random from $[dn]\setminus \left(\mathcal{E}_{t-1}\cup \{e_t\}\right)$, the set of all unexplored stubs after having removed $e_t$.
	\begin{itemize}
		\item [(a.1)] If $h_t\in \mathcal{U}_{t-1}$ and $J_t=1$, then all the \textit{unseen} stubs in the set $\mathcal{S}(h_t)\setminus \{h_t\}$ are declared active, while $e_t$ and $h_t$ are declared explored. In other terms, we update $\mathcal{A}_t\coloneqq \left(\mathcal{A}_{t-1}\setminus \{e_t\}\right)\cup\left(\mathcal{U}_{t-1} \cap \mathcal{S}(h_t)\setminus \{h_t\}\right)$, $\mathcal{U}_t\coloneqq \mathcal{U}_{t-1}\setminus \mathcal{S}(h_t)$ and $\mathcal{E}_t\coloneqq \mathcal{E}_{t-1}\cup\{e_t, h_t\}$.
		\item [(a.2)] If $h_t\in \mathcal{U}_{t-1}$ but $J_t=0$, then we simply declare $e_t$ and $h_t$ explored while the status of all other stubs remain unchanged. Thus we update $\mathcal{A}_t\coloneqq \mathcal{A}_{t-1}\setminus \{e_t\}$, $\mathcal{U}_t\coloneqq \mathcal{U}_{t-1}\setminus \{h_t\}$ and $\mathcal{E}_t\coloneqq \mathcal{E}_{t-1}\cup\{e_t,h_t\}$.
		\item [(a.3)] If $h_t\in \mathcal{A}_{t-1}$, then we simply declare $e_t$ and $h_t$ explored while the status of all other half-edges remain unchanged. In other terms, we update $\mathcal{A}_t\coloneqq \mathcal{A}_{t-1}\setminus \{e_t,h_t\}$, $\mathcal{U}_t\coloneqq \mathcal{U}_{t-1}$ and $\mathcal{E}_t\coloneqq \mathcal{E}_{t-1}\cup\{e_t,h_t\}$.
	\end{itemize}
	\item [(b)] If $|\mathcal{A}_{t-1}|=0$ and $|\mathcal{U}_{t-1}|\geq 1$, then we select a stub $e_t$ from $\mathcal{U}_{t-1}(\neq \emptyset)$ and we declare active \textit{all} the unseen stubs in $\mathcal{S}(e_t)$; note that at least one stub (namely $e_t$) is declared active. Then we proceed as in Step $(a)$.
		\item [(c)] If $|\mathcal{A}_{t-1}|=0$ and $|\mathcal{U}_{t-1}|=0$, then all the stubs have been paired and we terminate the algorithm.
\end{itemize}

\begin{oss}
	Note that, since in the above procedure we explore \textit{two half-edges} at each step, we have $\mathcal{A}_t\cup \mathcal{U}_t\neq \emptyset$ for $t\leq (dn/2)-1$ (recall that $dn$ is even) and $\mathcal{A}_{dn/2}\cup \mathcal{U}_{dn/2}=\emptyset$ (as $\mathcal{E}_{dn/2}=\{(u,i):u\in [n],i\in[d]\}$, the set of all stubs). Thus the algorithm runs for $dn/2$ steps.
\end{oss}

Note that, when $|\mathcal{A}_{t-1}|\geq 1$, we have
\begin{equation}\label{mainquantity}
\eta_t=\mathbb{1}_{\{h_t\in \mathcal{U}_{t-1}\}} J_t\left|\mathcal{S}(h_t)\cap \mathcal{U}_{t-1}\setminus \{h_t\}\right|-\mathbb{1}_{\{h_t\in \mathcal{A}_{t-1}\}}.
\end{equation}
In words, assuming that $|\mathcal{A}_{t-1}|\geq 1$, the number of active stubs at the end of step $t$ increases by $m-2\in \{0,1,\dots,d-2\}$ if $v(h_t)$ has $m\in\{2,\ldots,d\}$ unseen stubs at the end of step $t-1$ and $J_t=1$; it decreases by one if $h_t$ is unseen and $J_t=0$, or if $h_t$ is the unique unseen stub incident to $v(h_t)$; and it decreases by two if $h_t$ is an active stub. 

On the other hand, if $|\mathcal{A}_{t-1}|=0$ (but $|\mathcal{U}_{t-1}|\geq 1$), then (according to Step $(b)$ in algorithm \textbf{Generate $\mathbb{G}'(p)$}) we select a stub $e_t$ from $\mathcal{U}_{t-1}$, we declare active \textit{all} the $\left|\mathcal{S}(e_t)\cap \mathcal{U}_{t-1}\right|$ unseen stubs in $\mathcal{S}(e_t)$ and then we proceed as in Step $(a)$; hence in this case (using $e_t$ as the stub to be matched at time $t$) we see that $\eta_t$ equals
\begin{equation}\label{mainquantity2}
\left|\mathcal{S}(e_t)\cap \mathcal{U}_{t-1}\setminus \{e_t\}\right|+\mathbb{1}_{\{h_t\in \mathcal{U}_{t-1}\setminus \mathcal{S}(e_t) \}} J_t\left|\mathcal{S}(h_t)\cap \mathcal{U}_{t-1}\setminus \{h_t\}\right|-\mathbb{1}_{\{h_t\in \mathcal{S}(e_t)\}}.
\end{equation}
Next observe that, since at the end of each step $i$ in which $|\mathcal{A}_{i-1}|\geq 1$ we remove the (active) stub $e_i$ from $\mathcal{A}_{i-1}$ (after having matched it with an unpaired half-edge), we have the recursion:
\begin{itemize}
	\item $|\mathcal{A}_t|=|\mathcal{A}_{t-1}|+\eta_t-1$, if $|\mathcal{A}_{t-1}|\geq 1$;
	\item $|\mathcal{A}_t|=\eta_t$, if $|\mathcal{A}_{t-1}|=0$.
\end{itemize}

\subsection{Upper bound on $\mathbb{P}\left(|\CC_{\max}(\mathbb{G}(n,d,p))|<n^{2/3}/A\right)$}\label{upperboundGndp}
Set $R_t\coloneqq |\mathcal{A}_t|$ for $t\in [dn/2]\cup\{0\}$ and define $t_0\coloneqq 0$. Let $t_i\coloneqq \inf\{t\geq t_{i-1}+1:|\mathcal{A}_t|=0\}$ for $i\geq 1$ (as long as we do not enter step $(c)$ in the procedure \textbf{Generate} $\mathbb{G}(n,d,p)$). The next result, which corresponds to Lemma 10 in \cite{nachmias:critical_perco_rand_regular}, teaches us how to actually use the procedure \textbf{Generate $\mathbb{G}'(p)$} to study component sizes in $\mathbb{G}'(p)$. Specifically, it gives us a way to obtain an upper bound for the probability that $\CC_{\max}(\mathbb{G}'(p))$ contains less than $n^{2/3}/A$ vertices in terms of the (random) lengths of the positive excursions of the process $R_t$. 
\begin{lem}[Lemma 10 in \cite{nachmias:critical_perco_rand_regular}]\label{lemmarelcomp}
For $i\in \mathbb{N}$, we denote by $\mathcal{C}_i$ the $i$-th explored component in $\mathbb{G}'(p)$. Then $t_i-t_{j-1}\leq (d-1)|\CC_i|$ for every $i$ (until the end of the exploration process).
\end{lem}
In order to understand why Lemma \ref{lemmarelcomp} should hold, first of all note that the number of vertices in $\CC_i$ is given by (one plus) the number of steps $t\in (t_{i-1},t_j]$ at which $h_t$ is selected from the set of unseen stubs $\mathcal{U}_{t-1}$ \textit{and} the edge $e_th_t$ is retained in the percolation. Denote by $N_{UR}$ ($\leq t_i-t_{i-1}$) this random number of steps (where UR stands for \textit{unseen} and \textit{retained}) and denote by $N_{UNR}$ the (random) number of steps $t\in (t_{i-1},t_i]$ at which $h_t$ is unseen but the edge $e_th_t$ is \textit{not} retained in the percolation (here UNR stands for \textit{unseen} and \textit{not retained}). Then, with $N_A$ representing the number of steps $t\in (t_{i-1},t_i]$ at which $h_t$ is \textit{active}, we can decompose the $i$-th excursion length as $t_i-t_{i-1}=N_{UR}+N_{UNR}+N_A$. Moreover, the number of stubs that became active during the exploration of $\CC_i$ must be (approximately) equal to the total number of stubs that were added at any step between $t_{i-1}$ and $t_i$; denoting by $X$ such a random variable, this means that $0\approx -2N_A-N_{UNR}+X$, whence $2N_A+N_{UNR}\approx X$. Furthermore, at any given step during the exploration of $\CC_i$ we can add at most $d-2$ stubs to the set of active half-edges, so that $X\leq (d-2)N_{UR}$. Thus we conclude that $2N_A+N_{UNR}\lesssim (d-2)N_{UR}$, which together with the previous decomposition of $t_i-t_{i-1}$ gives us that $t_i-t_{i-1}\lesssim N_{UR}+(d-2)N_{UR}=(d-1)N_{UR}\approx (d-1)|\mathcal{C}_i|$, which is the desired bound. We refer the reader to \cite{nachmias:critical_perco_rand_regular} for the rigorous argument.

It follows from (\ref{switchgraph}) and Lemma \ref{lemmarelcomp} that, for all large enough $n$, setting $T\coloneqq \lceil n^{2/3}/A\rceil$ we can bound
\begin{align}\label{newprob}
\nonumber\mathbb{P}(|\CC_{\max}(\mathbb{G}(n,d,p))|<n^{2/3}/A)&\leq \mathbb{P}(|\CC_{\max}(\mathbb{G}(n,d,p))|<T)\\
\nonumber&\leq (2/c_d)\mathbb{P}(|\CC_{\max}(\mathbb{G}'_{p})|<T)\\
&\leq (2/c_d)\mathbb{P}(t_i-t_{i-1}<(d-1)T \text{ }\forall i).
\end{align}
Note that the probability on the right-hand side of (\ref{newprob}) is in the form we want, meaning that it equals (\ref{rightform}) (upon taking $N_1=(d-1)T$). Therefore, all we need to do in order to obtain an upper bound for such a quantity is to check that both conditions (H.1) and (H.2) are satisfied in the current setting. 

Before doing so, however, we state two simple facts related to the exploration process that we have used to generate the components of $\mathbb{G}'(p)$. 

The first result, which is taken from \cite{de_ambroggio_roberts:near_critical_RRG} and that we state here as a lemma but it is just an observation, provides lower and upper bounds on the number of vertices having $d$ unseen stubs and the number of active half-edges, respectively, at the end of any given step of the algorithm \textbf{Generate $\mathbb{G}'(p)$}. 

Throughout, we write $\mathcal{V}^{(k)}_i$ for the \textit{set} of vertices having $k\in [d]$ unseen half-edges at the end of step $i$.
\begin{lem}\label{boundnumbfreshandactive}
	For every step $t$ in the algorithm \textbf{Generate $\mathbb{G}'(p)$} we have 
	\[|\mathcal{V}^{(d)}_{t}|\geq n-1-2t \text{ and }R_t\leq d+2(d-1)t.\]
\end{lem}
The second result, instead, establishes two lower bounds for $\eta_t$ that we use in different contexts, namely during the exploration of a component ($R_{t-1}\geq 1$) or at the start of a new one ($R_{t-1}=0$).
\begin{lem}\label{lowercoupling}
	If $R_{t-1}\geq 1$, then
	\begin{equation}\label{lowereta1}
	\eta_t\geq J_{t}(d-1)-X^{(1)}_{t}-X^{(2)}_{t}\eqqcolon \eta'_{t},
	\end{equation}
	where we set
	\begin{equation*} 
	X^{(1)}_{t}\coloneqq (d-1)J_{t}\mathbb{1}_{\{v(h_{t})\in \cup_{k=1}^{d-1}\mathcal{V}^{(k)}_{t-1}\}} \text{ and }X^{(2)}_{t}\coloneqq \mathbb{1}_{\{h_{t}\in \mathcal{A}_{t-1}\}}\left(J_{t}(d-1)+1\right).
	\end{equation*}
	On the other hand, if $R_{t-1}=0$, then
	\begin{equation}\label{lowereta2}
	\eta_t\geq J_t(d-1)-J_t(d-1)\mathbb{1}_{\{v(h_t)\in \cup_{k=1}^{d-1}\mathcal{V}^{(k)}_{i-1} \cup \{v(e_i)\} \}}\eqqcolon \eta''_t.
	\end{equation}
\end{lem}
The proof of Lemma \ref{lowercoupling} is very simple and it follows from the expressions of $\eta_t$ given in (\ref{mainquantity}) and (\ref{mainquantity2}) for the cases where $R_{t-1}\geq 1$ and $R_{t-1}=0$, respectively. In simple words, $\eta'_t$ and $\eta''_t$ both coincide with $\eta_t$ if $h_t$ is selected from $\mathcal{V}^{(d)}_{t-1}$. The only difference between $\eta'_t$ and $\eta''_t$ is that the former random variable provides a lower bound for $\eta_t$ when the set of active stubs at the end of time $t-1$ is non-empty (meaning that we are currently exploring a component), whereas $\eta''_t$ do not involve the set of active stubs and we use it as a lower bound for $\eta_t$ when we start exploring a new component. 

\begin{oss}
	Intuitively, the reason why $\eta'_t$ and $\eta''_t$ provide good lower bounds on $\eta_t$ is because, in the time window in which we consider the exploration process (which consists of $\ll n$ many iterations), at any given time most of the vertices still possess $d$ unseen stubs at the start of the step; indeed, thanks to Lemma \ref{boundnumbfreshandactive} we know that $|\mathcal{V}^{(d)}_t|\geq n-1-2t=n(1-o(1))$. In other terms, in the time window under consideration we have not explored a large portion of the graph yet and $|\mathcal{V}^{(d)}_t|\approx n$. Thus, at most steps, the three random variables actually coincide. 
\end{oss}
Throughout the rest of this section we let $p=(d-1)^{-1},h>0$ and $T'\in \mathbb{N}$ be such that $T\leq T'=T'(n)\ll n$. (Later on we will take $h$ to be significantly smaller than $\sqrt{T'}$ but larger than $\sqrt{T}$.) 

Moreover, we let $\tau_h$ and $\tau_0$ be as in Definition \ref{defstop} with $T'$ and $T$ in place of $N_2$ and $N_1$, respectively. Furthermore, we take $\tau_1=\infty=\tau_2$ (so that trivially $\tau_h\wedge\tau_1=\tau_h$ and $\tau_0\wedge\tau_2=\tau_0$).

With the next two lemmas we show that conditions (H.1) and (H.2) are fulfilled. We remark that, in Lemma \ref{firstbitRRG}, the estimates are true for every $t\leq T'$ (and not only for all $t\leq \tau_h$). We start by checking (H.1).
\begin{lem}\label{firstbitRRG}
	For every $t\leq T'$ and for all large enough $n$: 
	\begin{itemize}
		\item $\mathbb{E}[\eta^2_{t}]\leq 4(d-1)^2$;
		\item $\mathbb{E}[\eta_t|R_{t-1}]\leq 1$ when $R_{t-1}\geq 1$; 
		\item $1-(12dT')/n\leq \mathbb{E}[\eta_t|R_{t-1}]$;
		\item $\mathbb{E}[\eta^2_t|R_{t-1}]\geq d-2-(24dT')/n$.
	\end{itemize}
\end{lem}
\begin{proof}
	Observe that, recalling the expressions of $\eta_t$ given in (\ref{mainquantity}) and (\ref{mainquantity2}) for the cases $R_{t-1}\geq 1$ and $R_{t-1}=0$, respectively, it is clear that, with $X=_d(d-1)\text{Ber}(p)$, we have
	\[\mathbb{E}[\eta^2_{t}]\leq \mathbb{E}[((d-1)+X)^2]\leq 4(d-1)^2.\]
	Next observe that, if $R_{t-1}\geq 1$, then clearly $\mathbb{E}[\eta_t|R_{t-1}]\leq\mathbb{E}[X]= (d-1)p=1$. There remains to compute the lower bounds on the (conditional) expectations of $\eta_t$ and $\eta^2_t$ given $R_{t-1}$. Recall from Lemma \ref{lowercoupling} that, if $R_{t-1}\geq 1$, we can bound $\eta_{t}\geq \eta'_{t}$. Thus recalling the definition of $\eta'_t$ we obtain
	\begin{equation}\label{gobackhere}
	\mathbb{E}[\eta_{t}|R_{t-1}]\geq \mathbb{E}[\eta'_{t}|R_{t-1}]=(d-1)p-\mathbb{E}[X^{(1)}_t|R_{t-1}]-\mathbb{E}[X^{(2)}_t|R_{t-1}].
	\end{equation}
	Now observe that, since at the start of any given step $t$ the number of unmatched stubs is $dn-2(t-1)-1$ (this is due to the fact that, at each step of the algorithm \textbf{Generate $\mathbb{G}'(p)$}, we pair \textit{exactly} two stubs), we have 
	\begin{equation*}
	\mathbb{E}[X^{(1)}_t|R_{t-1}]=(d-1)p\mathbb{P}(v(h_t)\in \cup_{k=1}^{d-1}\mathcal{V}^{(k)}_{t-1}|R_{t-1})\leq p\frac{(d-1)^2\sum_{k=1}^{d-1}|\mathcal{V}^{(k)}_{t-1}|}{dn-2(t-1)-1}
	\end{equation*}
	and also
	\begin{equation*}
	\mathbb{E}[X^{(2)}_t|R_{t-1}]\leq \frac{(d-1)R_{t-1}}{dn-2(t-1)-1}[p(d-1)+1].
	\end{equation*}
	Using Lemma \ref{boundnumbfreshandactive} we obtain $R_{t-1}<2(d-1)t$ (recall that $d\geq 3$) and, by ignoring vertices without unseen half-edges, we can bound (again thanks to Lemma \ref{boundnumbfreshandactive}) $\sum_{k=1}^{d-1}|\mathcal{V}^{(k)}_{t-1}|<2t$. Thus, for every $t\leq T'$, we obtain
	\begin{equation}\label{thetwox}
	\mathbb{E}[X^{(1)}_t|R_{t-1}]\leq \frac{p(d-1)^22T'}{dn-2(T'-1)-1}, \text{ }\mathbb{E}[X^{(2)}_t|R_{t-1}]\leq \frac{(d-1)^22T'}{dn-2(T'-1)-1}[p(d-1)+1].
	\end{equation}
	Using the bounds in (\ref{thetwox}) (and recalling that $T'\ll n,p=(d-1)^{-1}$) it is not difficult to show that, for all large enough $n$, we have $\mathbb{E}[X^{(1)}_t|R_{t-1}]+\mathbb{E}[X^{(2)}_t|R_{t-1}]\leq (12dT')/n$ and therefore, going back to (\ref{gobackhere}), we conclude that (when $R_{t-1}\geq 1$)
	\[1-\frac{12dT'}{n}\leq \mathbb{E}[\eta_t|R_{t-1}]\leq 1.\]
	When $R_{t-1}=0$, an almost identical computation reveals that (if $n$ is large enough) $\mathbb{E}[\eta_t|R_{t-1}]\geq 1-(12dT')/n$ too. There remains to bound (from below) the (conditional) second moment of $\eta_t$, given $R_{t-1}$. Consider first the case where $R_{t-1}\geq 1$. Note that $\eta_t$ is negative if, and only if, $\eta'_t$ is negative, in which case the two coincide and are equal to $-1$. Thus $\mathbb{E}[\eta^2_{t}|R_{t-1}]\geq \mathbb{E}[(\eta'_t)^2|R_{t-1}]$ and, using the bounds in (\ref{thetwox}), we see that the latter expectation is at least
	\begin{equation*}
	(d-1)^2p-2p(d-1)\mathbb{E}[X^{(1)}_t|R_{t-1}]-2p(d-1)\mathbb{E}[X^{(2)}_t|R_{t-1}]\geq d-1-\frac{24dT'}{n}.
	\end{equation*}
	Next, suppose that $R_{t-1}=0$. Then, by Lemma \ref{lowercoupling}, we know that $\eta_t\geq \eta''_t$ and also in this case we obtain $\mathbb{E}[\eta^2_t|R_{t-1}]\geq \mathbb{E}[(\eta''_t)^2|R_{t-1}]\geq d-1-(24dT')/n$, completing the proof of the lemma.
\end{proof}
We proceed by checking that also (H.2) is satisfied in the current setting.
\begin{lem}\label{secondbitRRG}
	For every $t\leq \tau_0$ and for all large enough $n$:
		\begin{itemize}
			\item $\mathbb{E}[\eta^2_{\tau_h+t}|\eta_{\tau_h},\dots,\eta_{\tau_h+t-1},\tau_h]\leq 4(d-1)^2$;
			\item $\mathbb{E}[\eta_{\tau_h+t}|\eta_{\tau_h},\dots,\eta_{\tau_h+t-1},\tau_h]\geq 1- (24dT')/n$ when $R_{\tau_h+t-1}<h$.
		\end{itemize}
\end{lem}
\begin{proof}
The first statement was shown in Lemma \ref{firstbitRRG} and similar computations to those carried out in the proof of that lemma establish the second statement.
\end{proof}
We are now in the position to show that $\mathcal{C}_{\max}(\mathbb{G}(n,d,p))$ contains more than $n^{2/3}/A$ vertices with probability at least $1-O_d(A^{-1/2})$. Indeed, let $T'\coloneqq \lceil n^{2/3}\rceil\geq T$. Then it follows from Proposition \ref{jointprop} together with (\ref{newprob}) that, setting $c_0=4(d-1)^2,\varepsilon_1=(12dT')/n,\sigma^2=d-1, \varepsilon_2=(24dT')/n,c_1=4(d-1)^2,\varepsilon_3=(24dT')/n$ and taking $h\coloneqq n^{1/3}/(24dA^{1/4})$,
\begin{equation*}
\mathbb{P}(|\CC_{\max}(\mathbb{G}(n,d,p))|<n^{2/3}/A)\leq \frac{C_2}{A^{1/2}}
\end{equation*}
for all sufficiently large $n$, for some constant $C_2=C_2(d)>0$ which depends solely on $d$.

\subsection{The $\mathbb{G}(n,m,p)$ model}\label{subsectionGnmp}
Recall that the $\mathbb{G}(n,m,p)$ random graph is constructed from a (random) bipartite graph $B(n,m,p)$ on two set of nodes, $V=\{v_1,\dots,v_n\}$ and $W=\{w_1,\dots,w_m\}$, called the set of vertices and attributes (or features), respectively, obtained by performing $p$-bond percolation on the complete bipartite graph on $V$ and $W$. In particular, the edge $\{v_i,v_j\}$ (with $i\neq j$) is present in $\mathbb{G}(n,m,p)$ if, and only if, there exists an attribute $w\in W$ such that both edges $\{v_i,w\}$ and $\{v_j,w\}$ are present in $B(n,m,p)$. Criticality is achieved when $p=1/\sqrt{nm}$. 

Recall that we are interested in the case where $m=\Theta(n)$; more precisely, we let $m=\lfloor \beta n\rfloor$ for some $\beta>0$ which is a model parameter.

As we did for the $\mathbb{G}(n,d,p)$ and $\mathbb{G}(n,p)$ random graphs, we start by describing an exploration process to reveal the connected components of the random intersection graph $\mathbb{G}(n,m,p)$.

\subsubsection{An exploration process}\label{explorationGnmp}
During the exploration process, each vertex in $V$ is \textit{active}, \textit{explored} or \textit{unseen}, whereas each attribute is \textit{discovered} or \textit{fresh}. At time $t=0$, we choose a node $v$ from $V$ and declare it active; all other vertices in $V$ are declared unseen. Also, at the start of the algorithm (i.e. at time $t=0$) all attributes in $W$ are declared fresh. At time $t=1$, we reveal all the (necessarily fresh) attributes linked to $v$ and declare these discovered; moreover, all the (necessarily unseen) vertices in $V$ which are linked to \textit{at least one} of the (discovered) attributes of $v$ are declared active, whereas $v$ itself becomes explored. At time $t=2$, we let $v'$ be one of the active vertices, if there is at least one such node; otherwise we select $v'$ from the set of unseen vertices. Then we reveal all the fresh attributes linked to $v'$ (if any) and declare these discovered; moreover, all the unseen nodes linked to \textit{at least one} of the attributes of $v'$ are declared active, whereas $v'$ itself becomes explored. The procedure continues until all vertices in $V$ are in status explored.

We now provide the formal description. Fix an ordering of the $n$ vertices in $V$ with $v$ first. Let us denote by $\mathcal{A}_{t}$, $\mathcal{U}_t$ and $\mathcal{E}_t$ the (random) sets of active, unseen and explored \textit{vertices} at the end of step $t\in \mathbb{N}_0$, respectively. Then, for every $t\in \mathbb{N}_0$, we have $V=\mathcal{A}_t\cup \mathcal{U}_t\cup \mathcal{E}_t$ (a disjoint union), so that in particular $\mathcal{U}_t=V\setminus(\mathcal{A}_t\cup \mathcal{E}_t)$ at each step $t$. Moreover, we denote by $\mathcal{D}_t$ the (random) sets of discovered \textit{attributes} by the end of step $t\in \mathbb{N}_0$, so that $W\setminus \mathcal{D}_t$ is the set of fresh attributes at the end of step $t$.\\

\textbf{Generate $\mathbb{G}(n,m,p)$}. At time $t=0$, vertex $v$ is declared active whereas all other vertices are declared unseen; moreover, all the attributes are fresh. Therefore $\mathcal{A}_0= \{v\}$, $\mathcal{E}_0=\emptyset$ and $\mathcal{D}_0= \emptyset$. For every $t\in \mathbb{N}$, the algorithm proceeds as follows. 
\begin{itemize}
	\item [(a)] If $|\mathcal{A}_{t-1}|\ge 1$, we let $u_t$ be the first active node (here and in what follows, the term \textit{first} refers to the ordering that we have fixed at the very beginning of the procedure).
	\item [(b)] If $|\mathcal{A}_{t-1}|=0$ and $|\mathcal{U}_{t-1}|\ge 1$, we let $u_t$ be the first unseen vertex.
	\item [(c)] if $|\mathcal{A}_{t-1}|=0=|\mathcal{U}_{t-1}|$ (so that $\mathcal{E}_{t-1}=V$), we halt the procedure.
\end{itemize}
Let us denote by $\mathcal{N}_t$ the (random) set of (fresh) attributes that are linked to $u_t$; formally,
\begin{equation}\label{eqcurlyn}
\mathcal{N}_t\coloneqq \{w\in W\setminus \mathcal{D}_{t-1}:w\sim u_t\}.
\end{equation}
Moreover, we denote by $\mathcal{R}_t$ the (random) set of unseen neighbours of $u_t$ that are linked to \textit{at least one} of the attributes in $\mathcal{N}_t$; formally, we set $R_t\coloneqq \{u\in \mathcal{U}_{t-1}\setminus\{u_t\}:u\sim w \text{ for some }w\in \mathcal{N}_t\}$
(and note that $\mathcal{U}_{t-1}\setminus \{u_t\} = \mathcal{U}_{t-1}$ if $\mathcal{A}_{t-1}\neq \emptyset$). Then we update 
\begin{align*}
\mathcal{U}_t\coloneqq  
\begin{cases}
\mathcal{U}_{t-1}\setminus \mathcal{R}_t,& \text{ if $|\mathcal{A}_{t-1}|\geq 1$}\\ 
\mathcal{U}_{t-1}\setminus (\mathcal{R}_t\cup \{u_t\}), & \text{ if $|\mathcal{A}_{t-1}|=0$}
\end{cases}
\end{align*}
as well as $\mathcal{A}_{t}\coloneqq (\mathcal{A}_{t-1}\setminus \{u_t\})\cup \mathcal{R}_t,\mathcal{E}_t\coloneqq \mathcal{E}_{t-1}\cup \{u_t\}$ and $\mathcal{D}_{t}\coloneqq \mathcal{D}_{t-1}\cup \mathcal{N}_t$.\\

\begin{oss}
	Since in the above procedure \textbf{Generate} $\mathbb{G}(n,m,p)$ we explore \textit{one vertex} at each step, we have $\mathcal{A}_t\cup \mathcal{U}_t\neq \emptyset$ for every $t\leq n-1$ and $\mathcal{A}_n\cup\mathcal{U}_n= \emptyset$ (as $\mathcal{E}_n=V)$. Thus the algorithm runs for $n$ steps. 
\end{oss}
Denoting by $\eta_t$ the (random) number of unseen vertices that we add to the set of active nodes at time $t$, since at the end of each step $i$ in which $|\mathcal{A}_{i-1}|\geq 1$ we remove the (active) vertex $u_i$ from $\mathcal{A}_{i-1}$ (after having revealed its unseen neighbours), we have the recursion
\begin{itemize}
	\item $|\mathcal{A}_t|=|\mathcal{A}_{t-1}|+\eta_t-1$ if $|\mathcal{A}_{t-1}|>0$;
	\item $|\mathcal{A}_t|=\eta_t$ if $|\mathcal{A}_{t-1}|=0$.
\end{itemize}
Let us denote by $R_t$ and $U_t$ the number of active and unseen vertices at the end of step $t\in [n]\cup\{0\}$, respectively, so that 
\[R_t=|\mathcal{A}_t| \text{ and } U_t=|\mathcal{U}_t|=n-t-R_t.\]
Then, conditional on $R_{t-1}$ \textit{and} $\mathcal{N}_t$, the random variable $\eta_t$ depends on the past until time $t-1$ only through $R_{t-1}$ and, in particular, setting $N_t\coloneqq |\mathcal{N}_t|$ we have
\begin{equation}\label{lawetagnmp}
\eta_t=_d\Bin(U_{t-1}-\mathbb{1}_{\{R_{t-1}=0\}},1-(1-p)^{N_t}) \text{ for }t\in [n].
\end{equation} 

\subsection{Upper bound on $\mathbb{P}(|\mathcal{C}_{\max}(\mathbb{G}(n,m,p))|<n^{2/3}/A)$}\label{upperboundGnmp}
Define $t_0\coloneqq 0$ and $t_i\coloneqq \min\{t\geq t_{i-1}+1:R_t=0\}$, for $i\geq 1$ (as long as we do not enter Step $(c)$ in the procedure \textbf{Generate $\mathbb{G}(n,m,p)$}). Then, denoting by $\mathcal{C}_i$ the $i$-th component revealed during the exploration process, we have $|\CC_i|=t_i-t_{i-1}$ for every $i$ and so, setting $T\coloneqq \lceil n^{2/3}/A\rceil$ we can write
\begin{equation}\label{excgnmp}
\mathbb{P}(|\mathcal{C}_{\max}(\mathbb{G}(n,m,p))|<n^{2/3}/A)\leq \mathbb{P}(|\mathcal{C}_{\max}(\mathbb{G}(n,m,p))|<T)=\mathbb{P}(t_i-t_{i-1}<T \text{ }\forall i).
\end{equation}
Note that the probability on the right-hand side of (\ref{excgnmp}) is in the form we want, meaning that it equals (\ref{rightform}) (upon taking $N_1=T$). Therefore, all we need to do in order to obtain an upper bound for such a quantity is to check that conditions (H.1) and (H.2) are satisfied also in this setting. 

Throughout the rest of this section we let $p=1/\sqrt{nm},h>0$ and $T'\in \mathbb{N}$ be such that $T\leq T'=T'(n)\ll n$. We also assume that $h\leq T'$. (Later on we will take $h$ to be significantly smaller than $\sqrt{T'}$ but larger than $\sqrt{T}$.) 

Moreover, we let $\tau_h$ and $\tau_0$ be as in definition \ref{defstop} with $T'$ and $T$ in place of $N_2$ and $N_1$, respectively. 

Contrary to what we did when analysing the random graphs $\mathbb{G}(n,p)$ and $\mathbb{G}(n,d,p)$, we do not set $\tau_i=\infty$ ($i\in \{1,2\}$); here the stopping times $\tau_1$ and $\tau_2$ are indeed used. More specifically, we use them to control the random variable $D_t\coloneqq |\mathcal{D}_t|$ throughout the time window under consideration. This in turn allows us to show that $N_t$ (which appears in the conditional law of $\eta_{t+1}$) is \textit{close} to a $\text{Bin}(m,p)$ distribution. 

Hence we define, for $x>0$:
\begin{equation}\label{auxstopp}
\tau_1=\tau_1(x)\coloneqq \inf\{t\in \mathbb{N}:D_t\geq x\} \text{ and } \tau_2=\tau_2(x)\coloneqq \inf\{t\in \mathbb{N}:D_{\tau_h+t}\geq x\}.
\end{equation} 
With the next two lemmas we show that the conditions (H.1) and (H.2) are fulfilled. We start by checking (H.1).
\begin{lem}\label{firstbitRIG}
	Let $x=3mpT' $ in the definition of $\tau_1,\tau_2$ given in (\ref{auxstopp}). There exist constants $c_1(\beta),c_2(\beta),c_3(\beta)>0$ that depend on $\beta$ such that, for every $t\leq \tau_h\wedge \tau_1$ and for all large enough $n$: 
\begin{itemize}
	\item $\mathbb{E}[\eta^2_{t}]\leq 2+c_1(\beta) $;
	\item $\mathbb{E}[\eta_t|R_{t-1}]\leq 1$; 
	\item $1-c_2(\beta)\frac{T'}{n}\leq \mathbb{E}[\eta_t|R_{t-1}]$;
	\item $\mathbb{E}[\eta^2_t|R_{t-1}]\geq 2-c_3(\beta)\frac{T'}{n}$.
\end{itemize}
\end{lem}
\begin{proof}
	Note that, since $(1-y)^M\geq 1-My$ for every $M\in \mathbb{N},y\geq -1$ (this is Bernoulli's inequality), we can bound $1-(1-p)^{N_t}\leq pN_t$. Thus, recalling (\ref{lawetagnmp}), we immediately see that 
	\begin{align*}
	\mathbb{E}[\eta^2_t]=\mathbb{E}[\mathbb{E}(\eta^2_t|R_{t-1},N_t)]&\leq \mathbb{E}[\mathbb{E}(\text{Bin}^2(n,pN_t)|N_t)]\\
	&\leq np\mathbb{E}[N_t]+(np)^2\mathbb{E}[N^2_t].
	\end{align*} 
	But clearly $N_t$ is stochastically dominated by the $\text{Bin}(m,p)$ distribution, whence $\mathbb{E}[N_t]\leq mp,\mathbb{E}[N^2_t]\leq mp+(mp)^2$ and an easy computation reveals that (since $p=1/\sqrt{nm},m=\lfloor \beta n\rfloor$)
	\[\mathbb{E}[\eta^2_t]\leq 2+(m/n)^{1/2}\leq 2+\beta^{1/2}=2+c_1(\beta).\]
	Similarly, $\mathbb{E}[\eta_t|R_{t-1}]\leq 1$. Next we bound from below the (conditional) mean of $\eta_t$ given $R_{t-1}$. Recalling (\ref{lawetagnmp}) we see that (since $t\leq \tau_h$ and $\tau_h$ is at most $T'$)
	\begin{align}\label{aftercomp}
	\nonumber\mathbb{E}[\eta_t|R_{t-1}]&=\mathbb{E}[\mathbb{E}(\Bin(U_{t-1}-\mathbb{1}_{\{R_{t-1}=0\}},1-(1-p)^{N_t})|R_{t-1},N_t)|R_{t-1}]\\
	\nonumber&\geq \mathbb{E}[(n-T'-R_{t-1})(1-(1-p)^{N_t})|R_{t-1}]\\
	&\geq (n-T'-h)\mathbb{E}[pN_t-(pN_t)^2/2|R_{t-1}],
	\end{align}
	where the last inequality follows from the classical bound $e^{-x}\leq 1-x+x^2/2$ (which is valid for all $x\geq 0$). Now since $h\leq T'$, $\mathbb{E}[N^2_t|R_{t-1}]\leq mp+(mp)^2$ and  $\mathbb{E}[N_t|R_{t-1}]=(m-D_{t-1})p\geq (m-x)p$ (as $t\leq \tau_1$), a simple computation then shows that the expression on the right-hand side of (\ref{aftercomp}) is at least $1-(c_2(\beta)T')/n$, for some constant $c_2(\beta)>0$ which depends only on $\beta$. There remains to bound from below the (conditional) mean of $\eta^2_t$ given $R_{t-1}$. We have
\begin{align*}
\mathbb{E}[\eta^2_t|R_{t-1}]&\geq \mathbb{E}[\mathbb{E}(\Bin^2(n-2T',1-(1-p)^{N_t})|R_{t-1},N_t)|R_{t-1}]\\
&=\mathbb{E}[\mathbb{V}(\Bin(n-2T',1-(1-p)^{N_t})|R_{t-1},N_t)|R_{t-1}]\\
&\hspace{1cm}+\mathbb{E}[\mathbb{E}(\Bin(n-2T',1-(1-p)^{N_t})|R_{t-1},N_t)^2|R_{t-1}].
\end{align*}
Using Jensen's inequality to control the second term on the right-hand side of the last expression and performing similar computations to those required to establish the above-mentioned lower bound on $\mathbb{E}[\eta_t|R_{t-1}]$, we arrive at $\mathbb{E}[\eta^2_t|R_{t-1}]\geq 2-(c_3(\beta)T')/n$ for some constant $c_3(\beta)>0$ which depends solely on $\beta$, thus completing the proof of the lemma.
\end{proof}
We proceed by checking that also (H.2) is satisfied in the current setting.
\begin{lem}\label{secondbitRIG}
	Let $x=3mpT' $ in the definition of $\tau_1,\tau_2$ given in (\ref{auxstopp}). There exists a constant $c_4(\beta)>0$ that depends on $\beta$ such that, for every $t\leq \tau_0\wedge\tau_2$ and for all large enough $n$:
	\begin{itemize}
		\item $\mathbb{E}[\eta^2_{\tau_h+t}|\eta_{\tau_h},\dots,\eta_{\tau_h+t-1},\tau_h]\leq 2+c_1(\beta)$;
		\item  $\mathbb{E}[\eta_{\tau_h+t}|\eta_{\tau_h},\dots,\eta_{\tau_h+t-1},\tau_h]\geq 1- c_4(\beta)\frac{T'}{n}$ when $R_{\tau_h+t-1}<h$.
	\end{itemize}
\end{lem}
\begin{proof}
	The first statement is immediate as in Lemma \ref{firstbitRIG} and the second statement follows from similar computations as those required to establish the lower bound on $\mathbb{E}[\eta_t|R_{t-1}]$ in the previous lemma. 
\end{proof}
We are now in the position to show that $\mathcal{C}_{\max}(\mathbb{G}(n,m,p))$ contains more than $n^{2/3}/A$ vertices with probability at least $1-O(A^{-1/2})$. Indeed, let $T'\coloneqq \lceil n^{2/3}\rceil\geq T$. Then it follows from Proposition \ref{jointprop} together with (\ref{excgnmp}) that, setting $c_0=2+c_1(\beta),\varepsilon_1=c_2(\beta)T'/n,\sigma^2=2, \varepsilon_2=c_3(\beta)T'/n,c_1=2+c_1(\beta),\varepsilon_3=c_4(\beta)T'/n$ and taking $h\coloneqq n^{1/3}/(\tilde{C}A^{1/4})$, with $\tilde{C}=\tilde{C}(\beta)>0$ some large enough constant that depends on $\beta$,
\begin{equation}\label{befcher}
\mathbb{P}(|\CC_{\max}(\mathbb{G}(n,m,p))|<n^{2/3}/A)\leq \frac{C'_3}{A^{1/2}}+ \mathbb{P}(\tau_1\leq T')+\frac{\mathbb{P}(\tau_2\leq T)}{1-O(A^{-1/2})-\mathbb{P}(\tau_1\leq T')}
\end{equation}
for all sufficiently large $n$, for some constant $C'_3=C'_3(\beta)>0$ which depends solely on $\beta$. Since $D_t$ is increasing in $t$ and it is stochastically dominated by a $\text{Bin}(mt,p)$ random variable (this is due to the fact that, at each step $i\leq t$, we discover at most $\text{Bin}(m,p)$ auxiliary vertices), we obtain (thanks to Lemma \ref{cher} and recalling that $x=3mpT'$)
\begin{equation*}
\mathbb{P}(\tau_1\leq T')\leq \mathbb{P}(\exists t\leq T':D_t>3mpT')\leq \mathbb{P}(D_{T'}>3mpT')\leq \mathbb{P}(\text{Bin}(mT',p)>3mpT')\leq  e^{-cn^{2/3}}
\end{equation*}
for some $c>0$. A similar bound holds for $\mathbb{P}(\tau_2\leq T)$, whence we conclude that the expression on the right-hand side of (\ref{befcher}) is at most
$2C'_3A^{-1/2}\leq C_3A^{-1/2}$.

\subsection{The $\mathbb{G}(n,\beta,\lambda)$ model}\label{subsectionGnbetalambda}
In this last section we consider the (critical) \textit{quantum} version of the Erd\H{o}s-R\'enyi random graph. We refer the reader to \cite{dembo_et_al:component_sizes_quantum_RG} and references therein for explanations on such terminology, moving on instead to the precise description of the model, as it is given in \cite{dembo_et_al:component_sizes_quantum_RG}.

Let $\mathbb{S}_{\beta}$ denote the circle of length $\beta>0$ and set $\mathbb{G}_{n,\beta}\coloneqq [n]\times \mathbb{S}_{\beta}$, so that to each element $v\in [n]$ we associate the copy $\mathbb{S}^v_{\beta}=\{v\}\times \mathbb{S}_{\beta}$ of $\mathbb{S}_{\beta}$. In this way, a point in $\mathbb{G}_{n,\beta}$ has \textit{two} coordinates: its site $v\in [n]$ and the time $t\in \mathbb{S}_{\beta}$. Then we create, within each $\mathbb{S}^v_{\beta}$, finitely many holes, according to independent Poisson point processes $(\mathcal{P}_v:v\in [n])$ of intensity $\lambda>0$. As a result, each punctured circle $\mathbb{S}^v_{\beta}\setminus \mathcal{P}_v$ can be expressed as the (disjoint) union of $k_v\in \mathbb{N}$ connected intervals, namely
\begin{equation}\label{decomp}
\mathbb{S}^v_{\beta}\setminus \mathcal{P}_v=\bigcup_{j=1}^{k_v}I^v_j,
\end{equation}
where $|\mathcal{P}_v|$ equals $k_v$ (unless $|\mathcal{P}_v|=0$, in which case $k_v=1$ and the circle $\mathbb{S}^v_{\beta}$ remains intact). 

We add edges in the following (random) manner. To each unordered pair $\{u,v\}$ of (distinct) vertices $u, v\in [n]$, we associate a further circle $\mathbb{S}^{u,v}_{\beta}$ of length $\beta$ and a Poisson point process $\mathcal{L}_{u,v}=\mathcal{L}_{v,u}$ on $\mathbb{S}^{u,v}_{\beta}$ with intensity $1/n$. The processes $\mathcal{L}_{u,v}$ are assumed to be independent for different $(u,v)$ and also independent of $(\mathcal{P}_z:z\in [n])$. 

Then, two intervals $I^u_j,I^v_{\ell}$ ($u\neq v$) of the decomposition (\ref{decomp}) are considered to be directly connected if there exists some $t\in \mathcal{L}_{u,v}$ such that $(u,t)\in I^u_j$ and $(v,t)\in I^v_{\ell}$. In simple words, two intervals $I^u_j$ and $I^v_{\ell}$ (with $u\neq v$) of the decomposition (\ref{decomp}) are linked by an edge if they both include a time $t$ at which the Poisson point process $\mathcal{L}_{u,v}$ (on $\mathbb{S}^{u,v}_{\beta}$) jumps.

The resulting random graph, which is obtained after connecting intervals in the way we have just described, is denoted by $\mathbb{G}_{n,\beta,\lambda}$.

As remarked in \cite{dembo_et_al:component_sizes_quantum_RG}, setting $\mathcal{P}\coloneqq \cup_{u\in [n]}\mathcal{P}_u$ (a \textit{finite} collection of points), the decomposition $\mathbb{G}_{n,\beta,\lambda}\setminus \mathcal{P}=\CC_1\cup\dots\cup\CC_N$ into maximal components is well-defined and, moreover, each point $x=(v,t)\in \mathbb{S}^v_{\beta}$ is almost surely not in $\mathcal{P}$. Hence, the notion of component $\CC(x)$ of $x$ is also well-defined and throughout the size of a component is the number of \textit{intervals} it contains.
\begin{oss}
	Note that, in the degenerate case where $\lambda=0$, there are no holes and we are back to the $\mathbb{G}_{n,p}$ random graph with
	\begin{equation*}
	p=\mathbb{P}(\mathbb{S}^u_{\beta} \text{ and }\mathbb{S}^v_{\beta} \text{ are directly connected })\\
	= 1-\mathbb{P}(\text{Exp}(1/n)>\beta)=1-e^{-\beta/n}.
	\end{equation*}
\end{oss}
It is convenient for us to work with the following equivalent description of the $\mathbb{G}_{n,\beta,\lambda}$ model \cite{dembo_et_al:component_sizes_quantum_RG}. Let $\lambda>0$ and assign to each vertex $v\in [n]$ a circle of length $\theta\coloneqq \lambda \beta>0$ with an independent rate one Poisson process $\mathcal{P}_v$ of holes on it. The links between each pair of punctured circles are created by means of iid Poisson processes of intensity $(\lambda n)^{-1}$ (which are also independent of the rate $1$ Poisson processes of holes). The resulting random graph, denoted by $\mathbb{G}_{n,\theta}$, has the same law as $\mathbb{G}_{n,\beta,\lambda}$ (see \cite{dembo_et_al:component_sizes_quantum_RG}).

As we did for the random graphs studied in previous subsections, we start by describing an algorithm to reveal the connected components of the $\mathbb{G}(n,\beta,\lambda)$ model. More precisely, we describe an exploration process that sequentially constructs an instance of this random graph, interval by interval.

\subsubsection{An exploration process}\label{explorationGnbetalambda}
In this description (which is taken from \cite{dembo_et_al:component_sizes_quantum_RG}) our vertices, which we recall are represented by circles of length $\theta$, have first been labelled with numbers $1,\dots,n$ and, at the end of each time step $t\in \mathbb{N}$ of the algorithm, exactly one interval becomes \textit{explored}.\\

\textbf{Generate $\mathbb{G}(n,\beta,\lambda)$}. At time $t=0$, we fix the vertex $w_0=1$ and choose a point $s_0$ uniformly at random on $\mathbb{S}^{w_0}_{\theta}$. The point $(w_0,s_0)$ is declared \textit{active}, whereas the remaining space is declared \textit{neutral}. Hence, denoting by $\mathcal{A}_t$ the set of active points at the end of step $t\in \mathbb{N}_0$, we have $|\mathcal{A}_0|=1$. For every $t\in \mathbb{N}$, the algorithm works as follows. 
\begin{itemize}
	\item [(a)] If $|\mathcal{A}_{t-1}|\geq 1$, we choose an active point $(w_t,s_k)$ whose vertex has the smallest index among all active points. In case of a tie (which may occur since each $\mathbb{S}^v_{\theta}$ could contain several active points), we choose the active point which chronologically appeared earlier than the others on the same vertex. 
	\item [(b)] If $|\mathcal{A}_{t-1}|=0$ and there exists at least one neutral circle, we choose $w_t$ to be the neutral vertex with the smallest index and select $s_t$ uniformly at random on $\mathbb{S}^{w_t}_{\theta}$. Then we declare the point $(w_t,s_t)$ active.
	\item [(c)] If $|\mathcal{A}_{t-1}|=0$ and there is no neutral circle left, we choose $w_t$ to be the vertex of smallest index among the vertices having some neutral part and select $s_t$ uniformly at random on the neutral part of $\mathbb{S}^{w_t}_{\theta}$. Then we declare the point $(w_t,s_t)$ active.
	\item [(d)] If $|\mathcal{A}_{t-1}|=0$ and there is no neutral part available on any circle, then we terminate the procedure.
\end{itemize}
The algorithm proceeds as follows. Using iid $\text{Exp}(1)$ random variables $J^t_-$ and $J^t_+$, we extract out of the \textit{maximal} neutral interval $\{w_t\}\times (s^1_t,s^2_t)$ around the (active) point $(w_t,s_t)$ the sub-interval $I_t\coloneqq \{w_t\}\times \tilde{I}_t$, where
\[\tilde{I}_t\coloneqq (s^1_t \vee (s_t-J^t_-),s^2_t\wedge (s_t+J^t_+)).\]
If $\mathbb{S}^{w_t}_{\theta}$ is neutral (apart from active points), we take $s^1_t=-\infty=-s^2_t$ (so that $\tilde{I}_t= (s_t-J^t_-,s_t+J^t_+)$) and $\tilde{I}_t=\mathbb{S}_{\theta}$ (the whole circle of length $\theta$) if $J^t_-+J^t_+\geq \theta$. We then remove from the list of active points $\mathcal{A}_{t-1}$ \textit{all} those points which got encompassed by the interval $I_t$, including the active point currently under investigation $(w_t,s_t)$. The links in the graph connected to all such points, other than  $(w_t,s_t)$, are considered to be \textit{surplus edges}. 

The connections of $I_t$ are then constructed in the following manner. 
\begin{itemize}
	\item [(i)] For $i\neq w_t$, we regard $\tilde{I}_t$ as a subset of $\mathbb{S}^{w_t,i}_{\theta}$ and sample the process of links $\mathcal{L}_{w_t,i}$ for times $r$ restricted to $\tilde{I}_t$. That is, we run the process $(\mathcal{L}_{w_t,i}(r))_{r\in \tilde{I}_t}$ and, denoting by $j^i_1,\dots,j^i_{L_i}\in \tilde{I}_t$ the jumps of this process (if any), then we create a link between each pair of points $(w_t,j^i_{\ell}),(i,j^i_{\ell})$, $\ell\in [L_i]$.
	\item [(ii)] Subsequently, we erase all links between $I_t$ and points on already explored intervals, and record each $(i,j^i_{\ell})$ ($\ell\in [L_i], i\neq w_t$) on the neutral space as an active point, labelled with the time (order) of its registration. 
\end{itemize}
After examining all the connections from $I_t$, we declare such interval \textit{explored} and increase $t$ by one. (We notice that the procedure halts after finitely many steps, because the number of intervals is finite.)\\

Denote by $\zeta_t$ the \textit{new} links that are formed during step $t$ in \textbf{Algorithm 3}. Moreover, we let $\text{Spl}_t$ denote the number of surplus edges found by the end of the first $t$ steps in the procedure. Define 
\begin{equation}\label{etaqrg}
\eta_t\coloneqq \zeta_t-(\text{Spl}_t-\text{Spl}_{t-1})
\end{equation}
and note that, since at those times $t$ where $|\mathcal{A}_{t-1}|\geq 1$ we add to the list of active points the $\zeta_t$ new links found during step $t$ but we remove those (active) points which are encompassed by $I_t$ (including $(w_t,s_t)$), we have the recursion
\begin{itemize}
	\item $|\mathcal{A}_{t}|=|\mathcal{A}_{t-1}|+\eta_t-1$, if $|\mathcal{A}_{t-1}|\geq 1$;
	\item $|\mathcal{A}_{t}|=\eta_t$, if $|\mathcal{A}_{t-1}|=0$.
\end{itemize}
To bound (from above) the probability that $|\mathcal{C}_{\max}(\mathbb{G}(n,\beta,\lambda))$ contains less than $n^{2/3}/A$ nodes, we follow the strategy adopted in \cite{dembo_et_al:component_sizes_quantum_RG}. Specifically, we consider a more \textit{restrictive} exploration process, which leads to shorter (positive) excursions compared to those of $|\mathcal{A}_t|$. The advantage of such a new procedure is that it yields a simple (conditional) distribution for the random variable $\eta_i$, similar to the one that we have encountered when dealing with the random graphs $\mathbb{G}(n,p)$ and $\mathbb{G}(n,m,p)$.\\

\textbf{Reduced exploration process}. This procedure, after creating the first active point on each node $u\in [n]$, it \textit{voids} all space on that vertex apart from the relevant interval around that point, thus sequentially producing components with no more intervals than does the original exploration process. Specifically, at each step $t\in \mathbb{N}$ of this procedure, either $I_t$ is originated from a point that was active at the end of step $t-1$ or, if $|\mathcal{A}_{t-1}|=0$, it is formed from a point $(w_t,s_t)$ with $s_t$ selected uniformly at random on the completely neutral circle $\mathbb{S}^{w_t}_{\theta}$, if such a circle exists; otherwise, we halt the procedure. Moreover, this new (restrictive) exploration process keeps \textit{at most} one link from $I_t$ to any (as of yet) never visited (in particular, neutral) circle $\mathbb{S}^i_{\theta}$, erasing all other connections which are being formed during step $t$ by the original exploration process described in \textbf{Generate $\mathbb{G}(n,\beta,\lambda)$}.\\

\begin{oss}
	Since in the above procedure we void all space on a node after creating the first active point on it, we see that the algorithm runs for $n$ steps.
\end{oss}

It is important to notice that such restrictive exploration process \textit{ does not} have surplus links 
and, moreover, its number of active points $|\mathcal{A}^*_t|$ also satisfies $|\mathcal{A}^*_0|=1$ and a recursion of the type:
\begin{itemize}
	\item $|\mathcal{A}^*_{t}|=|\mathcal{A}^*_{t-1}|+\eta^*_t-1$, if $|\mathcal{A}^*_{t-1}|\geq 1$;
	\item $|\mathcal{A}^*_{t}|=\eta^*_t$, if $|\mathcal{A}^*_{t-1}|=0$,
\end{itemize}
with $\eta^*_t$ representing the number of new active points created at time $t$. 

Let us denote by $R_t$ and $U_t$ the number of active points and neutral vertices at the end of step $t\in [n]\cup\{0\}$ in the reduced exploration process, respectively, so that
\[R_t=|\mathcal{A}^*_t| \text{ and }U_t=n-t-R_t.\]
Then, conditional on $\mathcal{F}_{t-1}$ \textit{and} $J_t$, the random variable $\eta^*_t$ depends on the past until time $t-1$ only through $R_{t-1}$ and, in particular, we have
\begin{equation}\label{lawetagntheta}
\eta^*_t=_d\text{Bin}(U_{t-1}-\mathbb{1}_{\{R_{t-1}=0\}}, 1-e^{-J_t/(\lambda n)}),
\end{equation}
where we recall that the $J_i$ are iid $\Gamma_{\theta}(2,1)$-distributed random variables.

\subsubsection{Upper bound on $\mathbb{P}(|\mathcal{C}_{\max}(\mathbb{G}(n,\beta,\lambda))|<n^{2/3}/A)$}\label{upperboundGnbetalambda}
Define $t_0\coloneqq 0$. Let $t_i\coloneqq \min\{t\geq t_{i-1}+1::R_t=0\}$ for $i\geq 1$ (as long as there is at least one neutral circle). Then, denoting by $\mathcal{C}_i$ the $i$-th component revealed during the \textit{reduced} exploration process, we have $|\mathcal{C}_i|=t_i-t_{i-1}$ for every $i$ and so, setting $T\coloneqq \lceil n^{2/3}/A\rceil$, we can write
\begin{equation}\label{firstboundGnbetalambda}
\mathbb{P}(|\CC_{\max}(\mathbb{G}_{n,\beta,\lambda})|<n^{2/3}/A)\leq \mathbb{P}(|\CC_{\max}(\mathbb{G}_{n,\beta,\lambda})|<T)\leq  \mathbb{P}(t_i-t_{i-1}<T \text{ }\forall i).
\end{equation}
Note that the probability on the right-hand side of (\ref{firstboundGnbetalambda}) is in the form we want, meaning that it equals (\ref{rightform}) (upon taking $N_1=T$). Therefore, all we need to do in order to obtain an upper bound for such a quantity is to check that conditions (H.1) and (H.2) are satisfied in this setting. 

Throughout the rest of this section we let $h>0$ and $T'\in \mathbb{N}$ be such that $T\leq T'=T'(n)\ll n$. We also assume that $h\leq T'$. (Later on we will take $h$ to be significantly smaller than $\sqrt{T'}$ but larger than $\sqrt{T}$.)

Moreover, we let $\tau_h$ and $\tau_0$ be as in definition \ref{defstop} with $T'$ and $T$ in place of $N_2$ and $N_1$, respectively. Furthermore, we take $\tau_1=\infty=\tau_2$ (so that trivially $\tau_h\wedge\tau_1=\tau_h$ and $\tau_0\wedge\tau_2=\tau_0$). 

With the next two lemmas we show that conditions (H.1) and (H.2) are fulfilled. We start by checking (H.1). 
\begin{lem}\label{firstbitquantum}
	There exist constants $c_1(\beta,\lambda),c_2(\beta,\lambda),c_3(\beta,\lambda)$ that depend on $\beta$ and $\lambda$ such that. for every $t\leq \tau_h$ and for all large enough $n$: 
	\begin{itemize}
		\item $\mathbb{E}[(\eta^*_{t})^2]\leq 2+c_1(\beta,\lambda)$;
		\item $\mathbb{E}[\eta^*_t|\mathcal{F}_{t-1}]\leq 1$; 
		\item $1-c_2(\beta,\lambda)\frac{T'}{n}\leq \mathbb{E}[\eta^*_t|\mathcal{F}_{t-1}]$;
		\item $\mathbb{E}[(\eta^*_t)^2|\mathcal{F}_{t-1}]\geq 2-c_3(\beta,\lambda)\frac{T'}{n}$.
	\end{itemize}
\end{lem}
\begin{proof}
	The proof follows the same type of computations carried out in the proof of Proposition \ref{firstbitRIG} and so we omit the details.
\end{proof}
We proceed by checking that also (H.2) is satisfied in the current setting.
\begin{lem}\label{secondbitquantum}
	There exists a constant $c_4(\beta,\lambda)$ which depends on $\beta$ and $\lambda$ such that, for every $t\leq \tau_0$ and for all large enough $n$:
	\begin{itemize}
		\item [(i)]  $\mathbb{E}[(\eta^*_{\tau_h+t})^2|\eta^*_{\tau_h},\dots,\eta^*_{\tau_h+t-1},\tau_h]\leq 2+c_1(\beta,\lambda)$;
		\item [(ii)]  $\mathbb{E}[\eta^*_{\tau_h+t}|\eta^*_{\tau_h},\dots,\eta^*_{\tau_h+t-1},\tau_h]\geq 1-c_4(\beta,\lambda)\frac{T'}{n}$ when $R_{\tau_h+t-1}<h$.
	\end{itemize}
\end{lem}
\begin{proof}
	Also in this case we omit the details as the computations are basically the same as those needed to establish Proposition \ref{firstbitquantum}. 
\end{proof}
We are now in the position to show that $\mathcal{C}_{\max}(\mathbb{G}(n,\beta,\lambda))$ contains more than $n^{2/3}/A$ vertices with probability at least $1-O(A^{-1/2})$. Indeed, let $T'\coloneqq \lceil n^{2/3}\rceil\geq T$. Then it follows from Proposition \ref{jointprop} together with (\ref{excgnp}) that, setting $c_0=2+c_1(\beta,\lambda),\varepsilon_1=(c_2(\beta,\lambda)T')/n,\sigma^2=2, \varepsilon_2=(c_3(\beta,\lambda)T')/n,c_1=2+c_1(\beta,\lambda), \varepsilon_3=(c_4(\beta,\lambda)T')/n$ and taking $h\coloneqq n^{1/3}/(C'A^{1/4})$ with $C'=C'(\beta,\lambda)>0$ some large enough constant that depends on $\beta,\lambda$,
\begin{equation*}
\mathbb{P}(|\CC_{\max}(\mathbb{G}(n,\beta,\lambda))|<n^{2/3}/A)\leq \frac{C_4}{A^{1/2}}
\end{equation*}
for all sufficiently large $n$, for some constant $C_4=C_4(\beta,\lambda)>0$ which depends solely on $\beta$ and $\lambda$.


\section{Proofs of Theorem \ref{corollary}}\label{upperboundsection}
We now prove Theorem \ref{corollary} by means of Lemmas \ref{simplelem} and \ref{secondsimplelem}. We start by considering the $\mathbb{G}(n,p)$ model and continue with the analysis of the random graphs $\mathbb{G}(n,d,p),\mathbb{G}(n,m,p)$ and $\mathbb{G}(n,\beta,\lambda)$, in this order. 

\subsection{The $\mathbb{G}(n,p)$ model}
Consider the $\mathbb{G}(n,p)$ random graph with $p=p(n)= n^{-1}$. Using the exploration process of Subsection \ref{subsectionGnp} we see that, given $k\in \mathbb{N}$, the event $\{|\mathcal{C}(V_n)>k\}$ is the same as $\{|\mathcal{A}_t|>0\text{ }\forall t\leq k \}$. We can couple the $\eta_i$ with iid random variables $X_i=_d\text{Bin}(n,1/n)$ in such a way that
\begin{align*}
\mathbb{P}(|\mathcal{C}(V_n)|>k)=\mathbb{P}(|\mathcal{A}_t|>0\text{ }\forall t\leq k)&=\mathbb{P}(1+\sum_{i=1}^{t}(\eta_i-1)>0\text{ }\forall t\leq k)\\
&\leq \mathbb{P}(1+\sum_{i=1}^{t}(X_i-1)>0\text{ }\forall t\leq k).
\end{align*} 
Since the $X_i$ have mean one, finite variance and because $\mathbb{P}(X_1=2)\sim 1/(2e)$ we obtain, thanks to Lemma  \ref{secondsimplelem}, that 
\[\mathbb{P}(|\mathcal{C}_{\max}(\mathbb{G}(n,p))|>An^{2/3})\leq 10eA^{-3/2}=C_5A^{-3/2}.\]

\subsection{The $\mathbb{G}_d(p)$ model}
Here we let $\mathbb{G}_d$ be \textit{any} $d$-regular graph on $[n]$ and denote by $\mathbb{G}_d(p)$ the $p$-percolated version of $\mathbb{G}_d$, where $p= 1/(d-1)$ and $3\leq d=d(n)<n-1$ is allowed to depend on $n$.
We use an exploration process in which \textit{vertices} are sequentially explored one by one and where, at each step, the unseen neighbours of the node under examination are declared active whereas the vertex itself becomes explored. However, rather than proceeding as in \cite{de_ambroggio:component_sizes_crit_RGs}, where the exploration process was started from \textit{two} active vertices and that caused the stronger requirement $d\geq 4$ (recall that here we are only assuming $d\geq 3$), we use the same exploration process employed for revealing the connected components in the $\mathbb{G}(n,p)$ model. 

Since $\mathbb{G}_d$ is $d$-regular and edges are retained \textit{independently} with probability $p$, it is easy to see that the random variables $\eta_i$ can be coupled with independent random variables $X_i$ satisfying $\eta_i\leq X_i$ and such that $X_1$ has the $\text{Bin}(d,p)$ distribution, whereas the $X_i$ (for $i\geq 2$) have the $\text{Bin}(d-1,p)$ distribution.

Then, setting $S_t\coloneqq 1+\sum_{i=1}^{t}(X_i-1)$, we can bound
\begin{multline}\label{newbin}
\mathbb{P}(|\mathcal{C}(V_n)|>k)=\mathbb{P}(|\mathcal{A}_t|>0\text{ }\forall t\leq k)\\
=\mathbb{P}(1+\sum_{i=1}^{t}(\eta_i-1)>0\text{ }\forall t\leq k)\leq \mathbb{P}\left(S_t>0\text{ }\forall t\leq k\right).
\end{multline}
In order to apply Lemma \ref{lemmaballot} (which requires a random walk with iid increments), we first need to substitute $X_1$ (which has the $\text{Bin}(d,p)$ distribution) with a random variable having the $\text{Bin}(d-1,p)$ distribution. To this end, first of all observe that we can assume each $X_1$ to be of the form $X_1=\sum_{j=1}^{d}I_j$, with $(I_j)_{j\in [d]}$ a sequence of iid random variables (also independent of the $X_i$, $i\geq 2$) having the $\text{Ber}(p)$ distribution. Then we can rewrite the probability on the right-hand side of (\ref{newbin}) as
\begin{align*}
\mathbb{P}(S_t>0\text{ }\forall t\in [k])&=\mathbb{P}(I_d+\sum_{j=1}^{d-1}I_j+\sum_{i=2}^{t}(X_i-1)>0\text{ }\forall t\in [k])\\
&=\mathbb{P}(1+I_d+\left(\sum_{j=1}^{d-1}I_j-1\right)+\sum_{i=2}^{t}(X_i-1)>0\text{ }\forall t\in [k])\\
&=\mathbb{P}(1+I_d+\sum_{i=1}^{t}(X'_i-1)>0\text{ }\forall t\in [k]),
\end{align*}
where we set $X'_1\coloneqq \sum_{j=1}^{d-1}I_j$ and $X'_i\coloneqq X_i$ for $2\leq i\leq k$. Now let $(J_j)_{j\in [d-2]}$ be a sequence of iid random variables with the $\text{Ber}(p)$ distribution, independent of all other random quantities involved. Define $Y_0\coloneqq \sum_{j=1}^{d-2}J_j+I_d$. Then, since
\[\mathbb{P}\left(\sum_{j=1}^{d-2}J_j=1\right)=\mathbb{P}(\text{Bin}(d-2,p)=1)=(d-2)p(1-p)^{d-3}\]
and $Y_0-I_d= \sum_{j=1}^{d-2}J_j$, using independence between the $I_j$ and the $J_j$ we can write
\begin{align}\label{WW}
\nonumber&\mathbb{P}(1+I_d+\sum_{i=1}^{t}(X'_i-1)>0\text{ }\forall t\in [k])\\
\nonumber&\hspace{1cm}=\frac{1}{(d-2)p(1-p)^{d-3}}\mathbb{P}(1+I_d+\sum_{i=1}^{t}(X'_i-1)>0\text{ }\forall t\in [k],Y_0-I_d=1)\\
&\hspace{1cm}\leq \frac{1}{(d-2)p(1-p)^{d-3}}\mathbb{P}(Y_0+\sum_{i=1}^{t}(X'_i-1)>0\text{ }\forall t\in \{0\}\cup [k]).
\end{align}
Setting $X''_0\coloneqq Y_0, X''_i\coloneqq X'_i$ for $i\in [k]$ and $S'_t\coloneqq \sum_{i=0}^{t}(X''_i-1)$ for $t\in \{0\}\cup [k]$, we see that the expression in (\ref{WW}) is at most
\begin{equation}
\frac{1}{(d-2)p(1-p)^{d-3}}\mathbb{P}\left(1+S'_t>0\text{ }\forall t\in [k]\right)\eqqcolon c_0\mathbb{P}\left(1+S'_t>0\text{ }\forall t\in [k]\right).
\end{equation}
Summarizing, we have shown that
\begin{equation*}
\mathbb{P}(|\mathcal{C}(V_n)|>\lfloor An^{2/3}\rfloor)\leq C_d\mathbb{P}\left(1+S'_t>0\text{ }\forall t\in [k]\right).
\end{equation*}
Since each $X''_i$ has mean one, finite variance and because $\mathbb{P}(X''_1=2)=\binom{d-1}{2}p^2(1-p)^{d-3}$, we can use Lemma \ref{secondsimplelem} to conclude that
\[\mathbb{P}(|\mathcal{C}_{\max}(\mathbb{G}(n,d,p))|>An^{2/3})\leq 10A^{-3/2}\frac{(d-1)^2}{(d-2)^2}\left(1-\frac{1}{d-1}\right)^{-2(d-3)}=C_6A^{-3/2}.\]

\subsection{The $\mathbb{G}(n,m,p)$ model}
Consider the $\mathbb{G}(n,m,p)$ random graph with $p=p(n)= 1/\sqrt{nm}$. Using the exploration process of Subsection \ref{subsectionGnmp} we see that, given $k\in \mathbb{N}$, the event $\{|\mathcal{C}(V_n)>k\}$ is the same as $\{|\mathcal{A}_t|>0\text{ }\forall t\leq k \}$. We can couple the $\eta_i$ with independent random variables $X_i=_d\text{Bin}(n,Np)$, where $N$ has the $\text{Bin}(m,p)$ distribution, in such a way that
\begin{align*}
\mathbb{P}(|\mathcal{C}(V_n)|>k)=\mathbb{P}(|\mathcal{A}_t|>0\text{ }\forall t\leq k)&=\mathbb{P}(1+\sum_{i=1}^{t}(\eta_i-1)>0\text{ }\forall t\leq k)\\
&\leq \mathbb{P}(1+\sum_{i=1}^{t}(X_i-1)>0\text{ }\forall t\leq k).
\end{align*} 
As before, we can use Lemma  \ref{secondsimplelem} to conclude that 
\[\mathbb{P}(|\mathcal{C}_{\max}(\mathbb{G}(n,p))|>An^{2/3})\leq C_7A^{-3/2}\]
for some constant $C_7=C_7(\beta)>0$ which depends on $\beta$.

\subsection{The $\mathbb{G}(n,\beta,\lambda)$ model}
Consider the $\mathbb{G}(n,\beta,\lambda)$ random graph with $F(\beta,\lambda)=1$. Due to the rather different nature of this model, we do not proceed exactly as we did for the random graphs $\mathbb{G}(n,p),\mathbb{G}(n,d,p)$ and $\mathbb{G}(n,m,p)$, but we follow a slightly different route; the first part of the argument is reported as it appears in \cite{dembo_et_al:component_sizes_quantum_RG}.

Recall the (full) exploration process of Subsection \ref{subsectionGnbetalambda}. Moreover, recall that $k_v$ is the number of (disjoint) connected intervals in decomposition (\ref{decomp}) of the $\mathbb{S}^v_{\beta}$. We denote by $k_v(T)$ the number of intervals in vertex $v\in [n]$ which belong to components whose sizes exceed $T\in \mathbb{N}$, so that $k_v(T)\leq k_v$ by definition. Furthermore, we denote by $\mathbb{C}(v,e)$ the connected component containing $\mathbb{S}^v_{\beta}$ after \textit{erasing} all the holes created in $\mathbb{S}^v_{\beta}$ by the Poisson process $\mathcal{P}_v$. Denote my $\mathcal{H}(v,j,T)$ the event that the interval $I^v_j$ in the decomposition (\ref{decomp}) is in a cluster of size greater than $T$. Since the size of the component containing interval $I^v_j$ is at most $|\mathbb{C}(v,e)|+k_v-1$ we see that, if $k_v\vee |\mathcal{C}(v,e)|\leq T$, then $\mathbb{1}_{\mathcal{H}(v,j,2T)}=0$ and hence
\[k_v(2T)=\sum_{j=1}^{k_v}\mathbb{1}_{\mathcal{H}(v,j,2T)}\leq k_v(\mathbb{1}_{\{|\mathcal{C}(v,e)|>T \}}+\mathbb{1}_{\{k_v>T\}}).\] 
Therefore we can use Markov's inequality to bound
\begin{align}\label{firstineqsgnbetalambda}
\nonumber\mathbb{P}(|\mathcal{C}_{\max}(\mathbb{G}(n,\beta,\lambda))|>2T)&\leq \mathbb{P}(\sum_{v\in [n]}^{}k_v(2T)>2T)\\
\nonumber&\leq  \frac{1}{2T}\sum_{v\in [n]}^{}\mathbb{E}[k_v(\mathbb{1}_{\{|\mathcal{C}(v,e)|>T \}}+\mathbb{1}_{\{k_v>T\}})]\\
&=\frac{1}{2T}\sum_{v\in [n]}^{}\mathbb{E}[k_v\mathbb{1}_{\{|\mathcal{C}(v,e)|>T \}}]+\frac{1}{2T}\sum_{v\in [n]}^{}\mathbb{E}[k_v\mathbb{1}_{\{k_v>T \}}].
\end{align}
(We remark that there is a small mistake in equation $(2.4)$ of \cite{dembo_et_al:component_sizes_quantum_RG}, as the term $k_v$ within the expectation that appears in the sum on the right-hand side of (\ref{firstineqsgnbetalambda}) was forgotten; this cause no harm, see below.)
The expression on the right-hand side of (\ref{firstineqsgnbetalambda}) is an error term and we get rid of it as follows. First we notice that
\begin{align}\label{ZZ}
\frac{1}{2T}\sum_{v\in [n]}^{}\mathbb{E}[k_v\mathbb{1}_{\{k_v>T \}}]&=\frac{1}{2T}\sum_{v\in [n]}^{}[T\mathbb{P}(k_v>T)+\sum_{\ell >T}^{}\mathbb{P}(k_v\geq \ell)].
\end{align}
Then, by Markov's inequality, we immediately see that
\[T\mathbb{P}(k_v>T)+\sum_{\ell >T}^{}\mathbb{P}(k_v\geq \ell)\leq \frac{2\mathbb{E}[k^2_v]}{T},\]
so that then the expression in (\ref{ZZ}) is at most $(2n\mathbb{E}[k^2_v])/T^2$. Next we bound the (main) term of (\ref{firstineqsgnbetalambda}), namely $(2T)^{-1}\sum_{v\in [n]}^{}\mathbb{E}[k_v\mathbb{1}_{\{|\mathcal{C}(v,e)|>T \}}]$, by means of Lemma \ref{simplelem}. To this end, we need to upper bound the process arising from the algorithm \textbf{Generate $\mathbb{G}(n,\beta,\lambda)$} with a random walk satisfying the assumption of Lemma \ref{simplelem}. This is already done in \cite{dembo_et_al:component_sizes_quantum_RG}, where it is argued that 
\begin{equation}
\mathbb{P}(|\mathcal{C}(v,e)|>T)\leq \mathbb{P}(1+S_t>0 \text{ }\forall t\in [T]),
\end{equation}
with $S_t\coloneqq \sum_{i=1}^{t}(\xi_i-1)$, where the $\xi_i$ are independent random variables each following the $\text{Poisson}(J_i/\lambda)$ law conditioned on $J_i$, which are iid random variables with the $\Gamma_{\theta}(2,1)$ distribution. 

Since $k_v(=_d\text{Poisson}(\theta)\vee 1)$ and $\mathbb{1}_{\{|\mathcal{C}(v,e)|>T\}}$ are independent we arrive at 
\[\frac{1}{2T}\sum_{v\in [n]}^{}\mathbb{E}[k_v\mathbb{1}_{\{|\mathcal{C}(v,e)|>T \}}]=\frac{2n\mathbb{E}[\text{Poisson}(\theta)\vee 1]}{2T}\mathbb{P}(1+S_t>0 \text{ }\forall t\in [T]).\]
Our assumption that $F(\beta,\lambda)=1$ implies that $\mathbb{E}[J_i]=\lambda$ and so $\mathbb{E}[\xi_i]=1$. Therefore we can apply Lemma \ref{simplelem} to deduce that
\[\mathbb{P}(|\mathcal{C}_{\max}(\mathbb{G}(n,\beta,\lambda))|>2T)\leq C_8A^{-3/2}\]
for some constant $C_8=C_8(\beta,\lambda)>0$ which depends on $\beta,\lambda$.





\section*{Acknowledgements}
This work was initiated when UDA was a PhD student at the University of Bath, where he was funded by the Royal Society. This work was also partially supported by ERC Grant Agreement 772606-PTRCSP.





\bibliographystyle{plain}
\def\cprime{$'$}

\end{document}